\documentclass[12pt,reqno]{amsart}%
\usepackage{amsmath, amsfonts, amssymb, amsthm, amscd, amsbsy}
\usepackage{fancyhdr}
\usepackage[usenames,dvipsnames,svgnames,x11names,hyperref]{xcolor}
\usepackage{geometry}
\usepackage{graphicx}
\usepackage{hyperref}
\usepackage{cleveref}
\usepackage{amsmath}
\usepackage{amsfonts}
\usepackage{amssymb}
\usepackage{lineno}%
\usepackage[normalem]{ulem}

\providecommand{\U}[1]{\protect\rule{.1in}{.1in}}
\hypersetup{colorlinks,breaklinks,
	linkcolor={Fuchsia},
	citecolor={ForestGreen},
	urlcolor={NavyBlue}}
\theoremstyle{plain}
\newtheorem{theorem}{Theorem}[section]

\newtheorem{corollary}[theorem]{Corollary}

\newtheorem{lemma}[theorem]{Lemma}

\numberwithin{equation}{section}
\allowdisplaybreaks
\newtheorem{theorema}{Theorem}[section]

\newcommand{\mr}{\mathbb{R}}
\newcommand{\p}{\partial}
\renewcommand{\l}{\left}
\renewcommand{\r}{\right}
\newcommand{\var}{\operatorname{Var}_{\mu_A}}
\newcommand{\E}{\mathcal{E}}

\begin{document}
\title[ Stability of the sharp  Poincar\'e and Brascamp--Lieb stability]{ The stability of Sharp Poincar\'e  and Brascamp--Lieb inequalities with full monomial weights}
\author{Nguyen Lam}
\address{Nguyen Lam: School of Science and the Environment, Grenfell Campus, Memorial
University of Newfoundland, Corner Brook, NL A2H5G4, Canada}
\email{nlam@mun.ca}
\author{Guozhen Lu}
\address{Guozhen Lu: Department of Mathematics, University of Connecticut, Storrs, CT
06269, USA}
\email{guozhen.lu@uconn.edu}
\author{Andrey Russanov}
\address{Andrey Russanov: Department of Mathematics, University of Connecticut, Storrs,
CT 06269, USA}
\email{andrey.russanov@uconn.edu}

\begin{abstract}

 It was known that the Gaussian measures with monomial weights satisfy the 
  the sharp pointwise curvature-dimension condition $CD(1, \infty)$, and consequently they allow us to derive the sharp constant $1$ for the  Poincar\'e inequality only with partial monomial weights. 
 In this paper, 
we first establish the
Gaussian measures with full monomial weights satisfy the 
integrated curvature-dimension condition $ICD(2, \infty)$ which is sufficient for us to  prove the  sharp Poincaré inequalities with optimal constant $2$ and the stability of the Brascamp–Lieb inequality for full monomial Gaussian measures. Several of these results are extended to homogeneous Gaussian measures on convex cones.  Using the sharp Poincaré inequality and the hypercontractivity, we obtain an improved Beckner inequality that recover the sharp Poincaré and logarithmic Sobolev constants at the endpoints. By developing a Laguerre spectral decomposition  of the associated Ornstein–Uhlenbeck type generator, we derive an exact identity for the Poincaré deficit and obtain sharp gradient stability for the Poincar\'e inequality with full monomial weights. We further establish $L^2$ and weighted gradient stability estimates for the Brascamp–Lieb inequality with explicit constants when all monomial exponents exceed \(1\). The \(L^2\) estimate extends to log-concave homogeneous weights under a subharmonicity assumption. Finally, we obtain an improved integrated curvature-dimension bound under a scale-invariant Hessian condition and identify the sharp constant for the Poincaré inequality in the radial class.

\end{abstract}
\maketitle

\section{Introduction}

The monomial weight $x^A = x^{\alpha_1}_1 \dots x_N^{\alpha_N}$, where $\alpha_i \geq 0$, was introduced by Cabr\'e and Ros-Oton in \cite{CR, CRO} in the study of stable solutions and of isoperimetric and Sobolev inequalities. It is an example of a homogeneous weight on the convex cone
\begin{equation*}
    \mathbb{R}_{\ast}^{N}=\left\{  x\in\mathbb{R}^{N}:x_{i}>0 \quad\text{ if} \quad \alpha_{i}>0\right\}.
\end{equation*}
Here a weight $w$ on an open cone $\Sigma\subset\mr^N$ is homogeneous of degree $\alpha\geq0$ if $w(tx)=t^\alpha w(x)$ for all $t>0$ and $x\in\Sigma$. The monomial weight $x^A$ is homogeneous of degree $\alpha=\alpha_1+\dots+\alpha_N$. Isoperimetric inequalities with homogeneous weights on convex cones were proved by Cabr\'e, Ros-Oton, and Serra in \cite{CRS16}. Sobolev, isoperimetric, Trudinger--Moser, and Caffarelli--Kohn--Nirenberg inequalities for this class were developed in \cite{CRO, CRS16} (see also \cite{Cas, DNP19, AJLL23}), and the associated Hardy identities and inequalities, together with their stability estimates, in
\cite{DLL22, CFLL24}. See also \cite{FLL-ANS} for Hardy identities in other settings and \cite{DPH25, Wan22} for Dunkl weights.

The Gaussian analogue is the probability measure
\begin{equation*}
    d\mu_{A}=\frac{x^{A}e^{-\frac{1}{2}|x|^{2}}}{\int_{\mathbb{R}_{\ast}^{N}
}x^{A}e^{-\frac{1}{2}|x|^{2}}dx}dx
\end{equation*}
on $\mathbb{R}_{\ast}^{N}$. This is the measure associated to the potential $W = \frac{|x|^2}{2}-\log x^A$, and its diffusion generator is
\begin{equation*}
    L_A = \Delta - \nabla W\cdot\nabla=\sum_{i=1}^N\left(\p_{x_ix_i}-\l(x_i-\frac{\alpha_i}{x_i}\r)\p_{x_i}\right).
\end{equation*}
More generally, let $\Sigma\subset\mr^N$ be an open convex set and let $w\in C(\overline\Sigma)\cap C^2(\Sigma)$ be positive in $\Sigma$. In most part of the paper $\Sigma$ is an open convex cone. We write
\begin{equation*}
    d\mu_w=\frac{1}{Z_w}\,w(x)e^{-\frac{|x|^2}{2}}dx,\qquad Z_w=\int_\Sigma w(x)e^{-\frac{|x|^2}{2}}dx,\qquad L_w=\Delta-x\cdot\nabla+\nabla\log w\cdot\nabla.
\end{equation*}
The monomial case is $\Sigma=\mr_*^N$, $w=x^A$, $\mu_w=\mu_A$.
We denote by $W^{1,2}_A(\mr_*^N,\mu_A)$ the closure of $C_c^\infty(\overline{\mr_*^N})$ with respect to the norm given by $\|f\|^2=\int|f|^2d\mu_A+\int|\nabla f|^2d\mu_A$. Throughout the paper, $L_A$ denotes the self-adjoint operator associated with the closed form $\E(f)=\int_{\mr_*^N}|\nabla f|^2d\mu_A$ on $W^{1,2}_A(\mr_*^N,\mu_A)$ (the Neumann realization, see Section~\ref{sec:spectral}), and $P_t=e^{tL_A}$ is the associated semigroup. In the same way, $W^{1,2}(\Sigma,\mu_w)$ is the closure of $C_c^\infty(\overline\Sigma)$ with respect to the norm given by $\|f\|^2=\int_\Sigma|f|^2d\mu_w+\int_\Sigma|\nabla f|^2d\mu_w$.

For a diffusion operator $L=\Delta-\nabla W\cdot\nabla$ we use the carr\'e du champ operators
\begin{equation*}
    \Gamma(f)=|\nabla f|^2,\qquad \Gamma_2(f)=\frac12L\Gamma(f)-\nabla f\cdot\nabla Lf=\|\nabla^2f\|^2+\langle\nabla^2W\,\nabla f,\nabla f\rangle,
\end{equation*}
where the last identity is the Bochner formula. We say that $L$ satisfies the pointwise curvature-dimension condition $CD(\rho,\infty)$ if 
$$\Gamma_2(f)\geq\rho\,\Gamma(f),$$ 
and the integrated curvature-dimension condition $ICD(\rho,\infty)$ if 
$$\int\Gamma_2(f)\,d\mu\geq\rho\int\Gamma(f)\,d\mu.$$ 
We write $P(\rho)$ and $LS(C,0)$, respectively, for the Poincar\'e and logarithmic Sobolev inequalities
\begin{equation*}
    \rho\operatorname{Var}_\mu(f)\leq\int\Gamma(f)\,d\mu,\qquad \operatorname{Ent}_\mu(|f|^2)\leq2C\int\Gamma(f)\,d\mu.
\end{equation*} We refer to \cite{BE85} for the curvature-dimension criterion and to
\cite{BGL14, Wang05, FW-209} for the $\Gamma$-calculus and semigroup framework used throughout, and to \cite{CZ25} for eigenvalue estimates of drifted Laplacians under Bakry--\'Emery Ricci curvature bounds. The operator $L_A$ satisfies the pointwise Bakry--\'Emery condition $CD(1,\infty)$.

As in \cite{LLR25}, we say that the monomial weight $x^A$ is \emph{full} if $\alpha_i>0$ for all $i$, and \emph{partial} otherwise. This paper is motivated by our previous works \cite{LLR25, LLR26}. We recall their main results.

\begin{theorema}[\cite{LLR25}]\label{thm:A}
For every monomial weight $x^A$ and every $u\in W^{1,2}_A(\mr_*^N,\mu_A)$,
\begin{equation*}
    \int_{\mr_*^N}|\nabla u|^2d\mu_A\geq\int_{\mr_*^N}\left|u-\int_{\mr_*^N}u\,d\mu_A\right|^2d\mu_A,
\end{equation*}
and
\begin{multline*}
    \int_{\mr_*^N}|\nabla u|^2d\mu_A-\int_{\mr_*^N}\left|u-\int_{\mr_*^N}u\,d\mu_A\right|^2d\mu_A\\
    \geq\frac12\int_{\mr_*^N}\left|\nabla\left(u-\left\langle\int_{\mr_*^N}\left(u-\int_{\mr_*^N}u\,d\mu_A\right)x\,d\mu_A,\,x\right\rangle\right)\right|^2d\mu_A.
\end{multline*}
If $x^A$ is partial, both inequalities are attained by nonconstant functions.
\end{theorema}
\begin{theorema}[\cite{LLR26}]\label{thm:B}
Let $w$ be positive on an open convex set $\Sigma$, with $-\nabla^2\log w\geq K_w\mathbb{I}$ on $\Sigma$ for some $K_w>-1$, and let $\mu_w$ be as above. Then for every $f\in W^{1,2}(\Sigma,\mu_w)$ and $1\leq p<2$,
\begin{equation*}
    (1+K_w)\operatorname{Var}_{\mu_w}(f)\leq\int_\Sigma|\nabla f|^2d\mu_w,\qquad \frac{\|f\|_2^2-\|f\|_p^2}{2-p}\leq\frac{1}{1+K_w}\int_\Sigma|\nabla f|^2d\mu_w,
\end{equation*}
\begin{equation*}
    \operatorname{Ent}_{\mu_w}(|f|^2)\leq\frac{2}{1+K_w}\int_\Sigma|\nabla f|^2d\mu_w,
\end{equation*}
where the norms are in $L^2(\mu_w)$ and $L^p(\mu_w)$. The Poincar\'e inequality also has a gradient stability estimate as in Theorem~\ref{thm:A}.
\end{theorema}
 For log-concave homogeneous weights, and in particular for all monomial weights, one has $K_w=0$, so Theorems~\ref{thm:A} and \ref{thm:B} hold with the Gaussian constant $1$. For monomial weights, this constant is sharp in the Poincar\'e and Beckner inequalities, and the stability estimates of Theorems~\ref{thm:A} and \ref{thm:B} are genuine stability results, only when the weight is partial. Indeed, if $x^A$ does not depend on $x_j$, then $\mu_A$ contains a standard Gaussian factor in $x_j$, and the optimizers of the Poincar\'e inequality are the affine functions of $x_j$. When the weight is full, this Gaussian direction disappears. Then the constant $1$ in these inequalities is not optimal, they have no non-constant optimizers, and their stability estimates measure the distance to affine functions, which are not optimizers. This is a limitation of the method. Theorems~\ref{thm:A} and \ref{thm:B} are proved from the pointwise bound $\nabla^2W\geq(1+K_w)\mathbb{I}$, that is, from $CD(1+K_w,\infty)$. For the monomial weight, $\nabla^2W=\mathbb{I}+\operatorname{diag}\left(\frac{\alpha_i}{x_i^2}\right)\geq\mathbb{I}$, so $L_A$ satisfies $CD(1,\infty)$. But the constant $1$ in $CD(1,\infty)$ cannot be improved, even for full weights, because $\inf_{x_i>0}\frac{\alpha_i}{x_i^2}=0$. So no pointwise argument can improve the constant $1$ in Theorems~\ref{thm:A} and \ref{thm:B}. So the full monomial case needs a different approach, and it is the main subject of this paper. We show that for full weights the sharp Poincar\'e inequality is $\operatorname{Var}_{\mu_A}(f)\leq\frac12\int|\nabla f|^2d\mu_A$, so the constant improves from $1$ to $\frac12$, with optimizers $x_i^2-(\alpha_i+1)$, and that the  constant in the Beckner inequality improves from $1$ to $\frac p2$, and that the correct stability estimates measure the distance to the span of these quadratic optimizers. In particular, the deficit identity of Theorem~\ref{poincare deficit} and the stability results of Section~\ref{sec:spectral} are the full-weight counterparts of the stability results in \cite{LLR25, LLR26}. On the other hand, the logarithmic Sobolev constant $1$ is sharp for every monomial weight, including full weights (see Section~\ref{sec:poincare beckner}).

To go beyond the constant $1$, in the case of the full monomial weight, in other words when every $\alpha_i > 0$, we establish the integrated curvature-dimension condition $ICD(2,\infty)$, which is weaker than the pointwise curvature-dimension condition $CD(2,\infty)$. Namely, we will show that
\begin{equation}\label{eq:ICD intro}
    \int \Gamma_2(f) d\mu_A \geq 2 \int \Gamma(f) d\mu_A
\end{equation}
for every $f\in C^2(\overline{\mr_*^N})$ whose derivatives of order at most $2$ grow at most polynomially (see Theorem~\ref{thm:ICD}). This will in turn allow us to prove a sharp Poincar\'e inequality for the full monomial weight with a better constant and explicit optimizers. It was shown by Bakry, Gentil and Ledoux in \cite{BGL14} that an integrated condition is enough to prove the Poincar\'e inequality. More recently, Huguet used this approach in \cite{HUG24} to prove a sharp Poincaré inequality for generalized Cauchy measures from an integrated curvature-dimension criterion. In \cite{HUG26} he proved logarithmic Sobolev inequalities for these measures from a curvature-dimension criterion.

Once we establish $ICD(2,\infty)$, it allows us to prove a Beckner inequality with a better constant using the improvement in energy decay of the semigroup. The new Beckner inequality recovers the sharp Poincar\'e and the sharp logarithmic Sobolev constants at the endpoints.
We collect these results in the following theorem. Compare it with Theorem~\ref{thm:B} with $K_w=0$, where the Poincar\'e and Beckner constants are $1$.
\begin{theorem}\label{thm:main}
Assume $\alpha_i>0$ for all $i$. Then $\mu_A$ satisfies $ICD(2,\infty)$ in the form \eqref{eq:ICD intro}. Consequently, for every $f\in W^{1,2}_A(\mr_*^N,\mu_A)$,
\begin{equation*}
    \operatorname{Var}_{\mu_A}(f)\leq\frac12\int_{\mr_*^N}|\nabla f|^2d\mu_A,
\end{equation*}
with equality if and only if $f-\int_{\mr_*^N}f\,d\mu_A$ is a linear combination of $x_i^2-(\alpha_i+1)$, $i=1,\dots,N$. Moreover, for every $p\in[1,2)$,
\begin{equation*}
    \frac{\|f\|_2^2-\|f\|_p^2}{2-p}\leq\frac p2\int_{\mr_*^N}|\nabla f|^2d\mu_A,
\end{equation*}
where the norms are in $L^2(\mu_A)$ and $L^p(\mu_A)$. The constant $\frac p2$ is sharp at $p=1$ and as $p\to2$, where the inequality becomes the sharp logarithmic Sobolev inequality with constant $1$. For $1<p<2$ the sharp constant lies in $[\frac12,\frac p2]$ though we do not know its exact value.
\end{theorem}
Theorem~\ref{thm:main} is proved in Section~\ref{sec:poincare beckner}.

Furthermore, we will analyze the spectral decomposition of the operator $L_A$. We will show that, under a suitable transformation, the operator becomes a Laguerre operator which has a discrete spectrum given by even integers. Once the full spectral decomposition is known, we can further analyze the Poincar\'e deficit and give an exact identity for the remainder.

\begin{theorem}\label{poincare deficit}
Assume $\alpha_i > 0$ for all $i$. For $j\geq0$ let $\Pi_j$ be the orthogonal projection in $L^2(\mu_A)$ onto the eigenspace of $-L_A$ with eigenvalue $2j$ (see Section~\ref{sec:spectral}), and define the remainders
$$R_iu = u -\Pi_0u- \Pi_1u - \dots -\Pi_iu = \sum_{j =i+1}^\infty \Pi_ju, \qquad i\geq1. $$
Then for all $u \in W^{1,2}_A(\mr_*^N, \mu_A)$
\begin{equation}
        \int_{\mr_*^N} |\nabla u|^2  d\mu_A - 2\int_{\mr_*^N} \left|u - \int_{\mr_*^N} u  d\mu_A\right|^2 d\mu_A  = \sum_{i=1}^\infty b_i \int_{\mr_*^N} \left|\nabla R_iu\right|^2 d\mu_A, \label{eq:poincare deficit}%
\end{equation}
where $b_i = \frac{1}{i(i+1)}$.
\end{theorem}
In particular, this proves sharp gradient stability and an improved version of the Poincar\'e inequality for full weights, where the corresponding estimates of \cite{LLR25, LLR26} are not sharp.

Quantitative stability of a sharp functional inequality asks that the deficit control the distance to the set of optimizers, a program initiated for the Sobolev inequality in \cite{BE91}, after a question of Brezis and Lieb \cite{BL85}, and since carried out in gradient form \cite{FN19, FZ22} and with optimal dimensional dependence \cite{DEFFL}. For remainder terms, quantitative forms and sharp stability constants in Sobolev and Hardy--Littlewood--Sobolev inequalities, see also \cite{CFMP09, Kon23, CLT242, CLT243, Carlen}. See also \cite{CF13} for the Gagliardo--Nirenberg--Sobolev family, \cite{FIL16, CLT23} and \cite{CLT242} for logarithmic Sobolev inequalities, and \cite{CFLL24, F21, DLL26, DN26, LLLS26-ANS,
LLLS26, DLLN26} for the Heisenberg uncertainty principle and the Caffarelli-Kohn-Nirenberg inequalities. For the Poincaré inequality the optimizers are the constants together with the first nontrivial eigenspace, and the relevant results are those of \cite{CF20, CF26, Ser23} for log-concave measures and of \cite{LLR25, LLR26, DLLN26, LLLS26, LLLS26-ANS} for weighted Gaussian measures on cones, half-spaces and orthants. Theorem \ref{poincare deficit} turns this into an exact identity for $\mu_A$. Exact decompositions of the Poincar\'e deficit along the spectrum, with sharp $L^2$ and gradient stability hierarchies, were obtained by Lam and Lu \cite{LL26} on Euclidean balls and for the standard Gaussian measure.

Lastly, we will prove the $L^2$ stability of the Brascamp--Lieb inequality
\begin{equation}
    \operatorname{Var}_{\mu_A}(f) \leq \int_{\mr_*^N} \langle (\nabla^2 W)^{-1} \nabla f, \nabla f\rangle d\mu_A. \label{ineq:Brascamp Lieb}%
\end{equation}
For a probability measure $d\mu=Z^{-1}e^{-W}dx$ with $W$ convex (below $\mu=\mu_A$ or $\mu=\mu_w$) and $f\in L^2(\mu)$ with $\int\langle(\nabla^2W)^{-1}\nabla f,\nabla f\rangle\,d\mu<\infty$ we denote the Brascamp--Lieb deficit by
\begin{equation*}
    \delta_{BL}(f)=\int\langle(\nabla^2W)^{-1}\nabla f,\nabla f\rangle\,d\mu-\operatorname{Var}_{\mu}(f)\geq0 .
\end{equation*}
Equality holds for functions $f=\langle\nabla W,c\rangle$, where $c\in\mr^N$, whenever $\nabla W\in L^2(\mu)$. For $\mu=\mu_A$ this means $\alpha_i>1$ for all $i$ (see Lemma~\ref{lem:monomial hypotheses}). Note that $\langle\nabla W,c\rangle\notin W^{1,2}_A(\mr_*^N,\mu_A)$ when $c_i\neq0$ and $1<\alpha_i\leq3$, since $\p_i^2W=1+\alpha_i/x_i^2$, but $\delta_{BL}(\langle\nabla W,c\rangle)$ is still well defined.
The inequality \eqref{ineq:Brascamp Lieb} is due to Brascamp and Lieb \cite{BL76}. Bobkov and Ledoux \cite{BL00} derived it, together with the logarithmic Sobolev inequality, from the Brunn--Minkowski and Prékopa--Leindler inequalities. Reinforcements and dimensional refinements were obtained in \cite{Har08, VHN14}, asymmetric $L^p$--$L^q$ covariance versions in \cite{CCL13, VHN19}, an intertwining formulation in \cite{ABJ18}, and the extension to weighted Riemannian manifolds with boundary in \cite{KM17}. Quantitative forms with explicit extremizers appear already in \cite{CE17}, and stability in dimension one, in total variation and Kolmogorov distance, in \cite{Ser23}.

Sharp stability for the Prékopa--Leindler and Borell--Brascamp--Lieb inequalities was established in \cite{FHT25}, following \cite{BFR24}. Machado and Ramos \cite{MR26} used this to prove a uniform $L^1$ stability estimate for \eqref{ineq:Brascamp Lieb},
\begin{equation*}
    \mathrm{dist}_{L^1(\mu)}\left(f,\mathcal{O}_{BL}\right) \leq C_N\,\delta_{BL}(f)^{1/2},
\end{equation*}
where $\mathcal{O}_{BL}=\{\langle\nabla W,c\rangle:c\in\mr^N\}$ and $C_N$ is independent of the potential, and showed that no uniform estimate of this kind holds in $L^p$ for any $p>1$. Böröczky, Qiu and Roberto \cite{BQR26} bootstrap this to $L^2$ under an additional hypothesis, which they call the super-Brascamp--Lieb inequality.

We follow the $L^2$ duality approach used, for instance, in \cite{FIL16, LLR25}. In particular, we show the following.
\begin{theorem}[Stability of Brascamp-Lieb inequality] \label{thm:BL stability} Let $\alpha_i > 1$ for all $i$ and $W = \frac{|x|^2}{2} - \log x^A$. Then for $f\in W^{1,2}_A(\mr_*^N,\mu_A)$ such that $\int_{\mr_*^N} f d\mu_A = 0$ we have
\begin{equation}
    \delta_{BL}(f) \geq \kappa_A \inf_{c\in\mr^N} \int_{\mr_*^N} \left| f - \left\langle \nabla W, c\right\rangle\right|^2 d\mu_A, \label{eq:BL Stability}
\end{equation}
where $\kappa_A = \left(\max_{i}\left[\frac{1}{2} + \alpha_i\left(\frac{1+\sqrt{\frac{\alpha_i+1}{2}}}{\alpha_i-1}\right)^2\right]\right)^{-1}$.
\end{theorem}
We extend this result to a more general homogeneous setting under several additional assumptions, which in the monomial case reduce to the condition $\alpha_i >1$ (see Lemma~\ref{lem:monomial hypotheses}). Let $\Sigma\subset\mr^N$ be an open convex cone and let $w\in C(\overline\Sigma)\cap C^2(\Sigma)$ be positive in $\Sigma$ and homogeneous of degree $\alpha>0$. Let $\mu_w$, $L_w$ and $W^{1,2}(\Sigma,\mu_w)$ be as above. 

\begin{theorem}[ Stability of Brascamp-Lieb inequality with homogeneous weights]
\label{thm:BL stability homogeneous}
    Let $w$ be homogeneous of degree $\alpha>0$, log-concave, on an open convex cone $\Sigma=\{w>0\}\subset\mr^N$, with $w\in C(\overline\Sigma)\cap C^2(\Sigma)$ and
    $w=0$ on $\p\Sigma$. Write $\psi=\log w$, $W=\frac{|x|^2}{2}-\psi$
    and $B=\nabla^2W=\mathbb{I}-\nabla^2\psi$. Assume also
    \begin{equation}
        \Delta(w^{\theta}) \geq 0\qquad  \text{for some }\theta\in(0,1).
        \label{eq:BLstab hypotheses}
    \end{equation}
    Together with the other assumptions, \eqref{eq:BLstab hypotheses} implies $\nabla\psi\in L^2(\mu_w)$, see Lemma~\ref{lemma:IBPw identity}.
    Then for every $f\in W^{1,2}(\Sigma,\mu_w)$ with
    $\int_\Sigma f\,d\mu_w=0$,
    \begin{equation}
        \delta_{BL}(f) \geq \kappa_w \inf_{c\in\mr^N} \int_\Sigma|f-\langle\nabla W,c\rangle|^2d\mu_w,
        \label{eq:BLstab conclusion}
    \end{equation}
    where
    \begin{equation}
        \kappa_w=\left(\frac{1}{\lambda_w}+\theta \left(\frac{1+\sqrt{1+\frac{(1-\theta)\alpha}{\lambda_w}}}{1-\theta}\right)^2\right)^{-1}
        \label{eq:BLstab kappa}
    \end{equation}
    and $\lambda_w>0$ is any constant such that $\lambda_w\operatorname{Var}_{\mu_w}(g)\leq\int_\Sigma|\nabla g|^2d\mu_w$ for all $g\in W^{1,2}(\Sigma,\mu_w)$. By Theorem~\ref{thm:B} one can take $\lambda_w=1$, and for the full monomial weight one can take $\lambda_w=2$ by Theorem~\ref{thm:main}.
\end{theorem}
For $w=x^A$ and $N=1$, Theorem~\ref{thm:BL stability homogeneous} with $\lambda_w=2$ and $\theta=\frac1{\alpha_1}$ gives exactly $\kappa_w=\kappa_A$. For $N\geq2$ it gives a smaller constant than Theorem~\ref{thm:BL stability}, since it uses only $\theta=\max_i\frac1{\alpha_i}$ and the total degree $\alpha=\alpha_1+\dots+\alpha_N$, while the proof of Theorem~\ref{thm:BL stability} works one coordinate at a time.

For the remainder of the paper, we will assume that the monomial weight is full unless stated otherwise.

The paper is organized as follows. In \Cref{sec:poincare beckner}, we establish the condition $ICD(2,\infty)$ and use it to prove a sharp Poincar\'e inequality and a Beckner inequality with an improved constant. In \Cref{sec:spectral}, we transform the monomial operator into a Laguerre operator and describe its spectral properties. We use it to establish the stability of the Poincar\'e inequality and to give an exact remainder identity. In \Cref{sec:stability of BL}, we establish the $L^2$ stability of the Brascamp--Lieb inequality, first for the monomial Gaussian measure and then for homogeneous Gaussian measures, using the duality approach. In \Cref{sec:ICD for hom}, we discuss the $ICD$ condition for homogeneous Gaussian measures.

\section{Poincar\'e and Beckner inequalities\texorpdfstring{---Proof of Theorem~\ref{thm:main}}{---Proof of Theorem 1.1}}\label{sec:poincare beckner}
In this section we show that the full monomial Gaussian measure satisfies the integrated curvature-dimension condition $ICD(2,\infty)$. We will then use it to prove the sharp Poincar\'e inequality and a Beckner inequality with an improved constant. We begin with the following lemma.
\begin{lemma}\label{lemma1}%
Let $\alpha > 0$ and let $d\mu_\alpha=Z_\alpha^{-1}x^{\alpha}e^{-\frac{x^2}{2}}dx$ be the one-dimensional monomial probability measure on $(0,\infty)$, where $Z_\alpha=\int_0^\infty x^{\alpha}e^{-\frac{x^2}{2}}dx=2^{\frac{\alpha-1}{2}}\Gamma\left(\frac{\alpha+1}{2}\right)$.  Let $h\in C^1(0,\infty)$ satisfy the following decay properties:
\begin{equation}
    x^{\alpha-1}e^{-\frac{x^2}{2}}\,|h|^2 \longrightarrow 0
    \qquad \text{as } x\to 0 \text{ and as } x\to\infty.
    \label{eq:decay}
\end{equation}
Then we have that
\begin{equation}
    \int_0^\infty |h'|^2d\mu_\alpha
    + \alpha\int_0^\infty \frac{|h|^2}{x^2}d\mu_\alpha
    \ \geq\ \int_0^\infty |h|^2d\mu_\alpha
    \label{eq:ineq}
\end{equation}
with equality if and only if $h(x) = cx$ for some constant $c$.
\end{lemma}

\begin{proof}
If $\int_0^\infty |h'|^2d\mu_\alpha$ or $\int_0^\infty \frac{|h|^2}{x^2}d\mu_\alpha$ is infinite, there is nothing to prove. So we assume both are finite. Then the cross term below is finite by the Cauchy--Schwarz inequality, and the identity obtained by integrating by parts on $[\varepsilon,M]$ and letting $\varepsilon\to0$, $M\to\infty$ shows that $\int_0^\infty |h|^2d\mu_\alpha$ is finite too. Hence all integrals below are finite.

Since $|h' - h/x|^2 \geq 0$, integrating we get
\begin{equation}
    0 \ \leq\ \int_0^\infty \left|h' - \frac{h}{x}\right|^2 d\mu_\alpha
    = \int_0^\infty |h'|^2d\mu_\alpha
    - 2\int_0^\infty \frac{h\,h'}{x}d\mu_\alpha
    + \int_0^\infty \frac{|h|^2}{x^2}d\mu_\alpha.
    \label{eq:square}
\end{equation}
We integrate the cross term by parts. Using $2h\,h' = (|h|^2)'$,
\begin{align*}
    2\int_0^\infty \frac{h\,h'}{x}d\mu_\alpha
    &= Z_\alpha^{-1}\int_0^\infty \frac{(|h|^2)'}{x} x^{\alpha} e^{-\frac{x^2}{2}}dx\\
    &= Z_\alpha^{-1}\left[|h|^2x^{\alpha-1} e^{-\frac{x^2}{2}}\right]_0^\infty
    - Z_\alpha^{-1}\int_0^\infty |h|^2\left(x^{\alpha-1} e^{-\frac{x^2}{2}}\right)' dx.
\end{align*}
The boundary term vanishes
by the assumption \eqref{eq:decay}. For the remaining integral we get
\begin{align*}
    - Z_\alpha^{-1}\int_0^\infty |h|^2\left(x^{\alpha-1} e^{-\frac{x^2}{2}}\right)' dx &= -Z_\alpha^{-1}\int_0^\infty \left((\alpha - 1)x^{\alpha-2}e^{-\frac{x^2}{2}} + x^{\alpha-1}e^{-\frac{x^2}{2}} (-x)\right) |h|^2dx \\
    &= -\int_0^\infty \frac{\alpha-1}{x^2}|h|^2 d\mu_\alpha + \int_0^\infty |h|^2 d\mu_\alpha.
\end{align*}
Substituting this identity into \eqref{eq:square} gives
\begin{align*}
    0 &\ \leq\ \int_0^\infty |h'|^2d\mu_\alpha
    - \int_0^\infty |h|^2d\mu_\alpha
    + (\alpha - 1)\int_0^\infty \frac{|h|^2}{x^2}d\mu_\alpha
    + \int_0^\infty \frac{|h|^2}{x^2}d\mu_\alpha \\
    &\ =\ \int_0^\infty |h'|^2d\mu_\alpha
    + \alpha\int_0^\infty \frac{|h|^2}{x^2}d\mu_\alpha
    - \int_0^\infty |h|^2d\mu_\alpha,
\end{align*}
which is \eqref{eq:ineq}. Equality holds if and only if the integrand in
\eqref{eq:square} vanishes identically, i.e.\ $h' = h/x$.
This means $(h/x)'=0$, hence $h(x) = cx$.
\end{proof}

Next we establish $ICD(2,\infty)$.

\begin{theorem}\label{thm:ICD}
Assume $\alpha_i>0$ for all $i$. Then for every $f\in C^2(\overline{\mr_*^N})$ whose derivatives of order at most $2$ grow at most polynomially at infinity,
\begin{equation*}
    \int_{\mr_*^N}\Gamma_2(f)\,d\mu_A\geq 2\int_{\mr_*^N}\Gamma(f)\,d\mu_A ,
\end{equation*}
where the left-hand side may be $+\infty$. The constant $2$ is sharp: equality holds for $f=x_i^2$.
\end{theorem}

\begin{proof}
First consider it in dimension $N=1$. Recall that our operator is
$$L_\alpha = \p_{xx} - \left(x-\frac{\alpha}{x}\right) \p_x,$$ and
$$\Gamma(f) = |f'|^2, \qquad \Gamma_2(f) = |f''|^2 + \left(1+\frac{\alpha}{x^2} \right)|f'|^2 . $$
Our goal is to show:
$$\int_0^\infty |f''|^2 + \left(1+\frac{\alpha}{x^2} \right)|f'|^2  d\mu_\alpha \geq 2 \int_0^\infty |f'|^2 d\mu_\alpha,$$
which is equivalent to
$$\int_0^\infty |f''|^2 d\mu_\alpha + \int_0^\infty |f'|^2 d\mu_\alpha + \alpha \int_0^\infty \frac{|f'|^2}{x^2} d\mu_\alpha \geq 2 \int_0^\infty |f'|^2 d\mu_\alpha. $$
Simplifying we get
$$\int_0^\infty |f''|^2 d\mu_\alpha + \alpha \int_0^\infty \frac{|f'|^2}{x^2} d\mu_\alpha \geq  \int_0^\infty |f'|^2 d\mu_\alpha.$$
But this is exactly Lemma~\ref{lemma1} applied to $h = f'$, and this proves the claim in dimension $N=1$. \\

Consider now general dimension $N$, and assume $\alpha_i > 0$ for all $i$. Our operator is
$$L_A = \sum_{i=1}^N \p_{x_ix_i} - \left(x_i - \frac{\alpha_i}{x_i}\right)\p_{x_i},$$
with
$$\Gamma(f) = |\nabla f|^2, \qquad \Gamma_2(f) = \|\nabla^2 f\|^2 + \sum_{i=1}^N \left(1+\frac{\alpha_i}{x_i^2}\right)|\p_{x_i}f|^2.$$
We first establish, for each fixed  $i$, the inequality
\begin{equation}
    \int_{\mr_*^N} |f_{x_ix_i}|^2 + \left(\frac{\alpha_i}{x_i^2}\right)|f_{x_i}|^2  d\mu_A \geq \int_{\mr_*^N} |f_{x_i}|^2  d\mu_A. \label{eq:lemmamd}%
\end{equation}
Since $\mu_A $ is a product measure, split the coordinates into the variable $x_i$ and factor $\mu_A = \mu_{\alpha_i}\otimes \mu_{\neq i}$. By Tonelli's theorem, since the integrand is nonnegative,
$$\int_{\mr_*^N} |f_{x_ix_i}|^2 + \left(\frac{\alpha_i}{x_i^2}\right)|f_{x_i}|^2  d\mu_A = \int_{\mr_*^{N-1}}\left( \int_0^\infty |f_{x_ix_i}|^2 + \left(\frac{\alpha_i}{x_i^2}\right)|f_{x_i}|^2  d\mu_{\alpha_i}\right) d\mu_{\neq i}.$$
Now for each fixed $x_{\neq i}$, the slice $x_i \mapsto \p_{x_i}f(x_i, x_{\neq i})$ is a function of the single variable $x_i$. Setting $h(x_i) = \p_{x_i}f(x_i, x_{\neq i})$, so that $h'(x_i) = \p_{x_ix_i}f(x_i, x_{\neq i})$, we obtain by Lemma~\ref{lemma1}
$$\int_0^\infty |f_{x_ix_i}|^2 + \left(\frac{\alpha_i}{x_i^2}\right)|f_{x_i}|^2  d\mu_{\alpha_i} \geq \int_0^\infty |f_{x_i}|^2  d\mu_{\alpha_i}.$$
Lemma~\ref{lemma1} can be used here. The decay \eqref{eq:decay} at infinity follows from the polynomial growth of $h$. Near $x_i=0$ there are three cases. If $h(0)=0$, then $h(x_i)=O(x_i)$ and \eqref{eq:decay} holds. If $h(0)\neq0$ and $\alpha_i>1$, then \eqref{eq:decay} holds since $h$ is bounded. If $h(0)\neq0$ and $\alpha_i\leq1$, then $\int_0^1x_i^{\alpha_i-2}dx_i=\infty$, so $\int_0^\infty \frac{|h|^2}{x_i^2}d\mu_{\alpha_i}=\infty$ and the inequality is trivial.
Since this holds for every $x_{\neq i}$, integrating over $\mu_{\neq i}$ and applying Tonelli once more gives
$$\int_{\mr_*^N} |f_{x_ix_i}|^2 + \left(\frac{\alpha_i}{x_i^2}\right)|f_{x_i}|^2  d\mu_A \geq \int_{\mr_*^N} |f_{x_i}|^2  d\mu_A,$$
which is \eqref{eq:lemmamd}. With the inequality, we compute
\begin{align*}
    \int_{\mr_*^N} \Gamma_2(f) d\mu_A - \int_{\mr_*^N} |\nabla f|^2 d\mu_A &= \int_{\mr_*^N} \left( \|\nabla^2 f\|^2 + \sum_{i=1}^N \left(\frac{\alpha_i}{x_i^2}\right)|\p_{x_i}f|^2 \right) d\mu_A \\
    &\geq \int_{\mr_*^N} \left( \sum_{i=1}^N |f_{x_ix_i}|^2 + \sum_{i=1}^N \left(\frac{\alpha_i}{x_i^2}\right)|\p_{x_i}f|^2 \right) d\mu_A \\
    &= \sum_{i=1}^N \int_{\mr_*^N} |f_{x_ix_i}|^2 + \left(\frac{\alpha_i}{x_i^2}\right)|\p_{x_i}f|^2  d\mu_A \\
   \text{by \eqref{eq:lemmamd}} &\geq \sum_{i=1}^N \int_{\mr_*^N} |f_{x_i}|^2  d\mu_A \\
   &= \int_{\mr_*^N} |\nabla f|^2 d\mu_A.
\end{align*}
Rearranging we get the $ICD(2,\infty)$ condition in dimension $N$:
\begin{equation}
    \int_{\mr_*^N} \Gamma_2(f) d\mu_A \geq 2 \int_{\mr_*^N} |\nabla f|^2 d\mu_A .\label{ICD(2,infty)}%
\end{equation}
For $f=x_i^2$ we have $\Gamma_2(f)=4+4\left(1+\frac{\alpha_i}{x_i^2}\right)x_i^2$, so $\int\Gamma_2(f)\,d\mu_A=4+4(\alpha_i+1)+4\alpha_i=8(\alpha_i+1)=2\int|\nabla f|^2d\mu_A$, since $\int x_i^2d\mu_A=\alpha_i+1$. Hence the constant $2$ is sharp.
\end{proof}
Using Theorem~\ref{thm:ICD} and the standard semigroup argument, we can now conclude that the measure $\mu_A$ satisfies a Poincar\'e inequality.
\begin{theorem}\label{poincareineq}
Assume $\alpha_i>0$ for all $i$. Then for every $f\in W^{1,2}_A(\mr_*^N,\mu_A)$
\begin{equation*}
    \operatorname{Var}_{\mu_A}(f) \leq \frac{1}{2} \int_{\mr_*^N} |\nabla f|^2 d\mu_A.
\end{equation*}
The constant $\frac12$ is sharp. Equality holds for $\varphi_i =\frac{\alpha_i + 1}{2} - \frac{x_i^2}{2}$, $i=1,\dots,N$, and more precisely, equality holds if and only if $f-\int_{\mr_*^N}f\,d\mu_A$ belongs to the eigenspace of $-L_A$ with eigenvalue $2$ (see Section~\ref{sec:spectral}).
\end{theorem}
\begin{proof}
The inequality follows by a standard argument using the equivalence of
$ICD(\rho,\infty)$ and the Poincar\'e inequality $P(\rho)$, see
\cite{BGL14}. A direct proof, which needs no regularity of the semigroup, uses the spectral decomposition of Section~\ref{sec:spectral}: with $c_j=\|\Pi_jf\|_{L^2(\mu_A)}$ we have $\operatorname{Var}_{\mu_A}(f)=\sum_{j\geq1}c_j^2\leq\frac12\sum_{j\geq1}2j\,c_j^2=\frac12\E(f)$, with equality if and only if $c_j=0$ for all $j\geq2$. This also gives the characterization of the equality cases. For the sharpness, since $\mu_A$ is a product measure and $\varphi_i$ depends only on $x_i$, it is enough to consider $N=1$. Write $\alpha=\alpha_1$ and take $f = x^2 = -2\varphi_1  +(\alpha+1)$. Then
\begin{align*}
    \int_0^\infty x^\alpha e^{-\frac{x^2}{2}}dx &= \frac{1}{2} 2^{\frac{\alpha+1}{2}}\Gamma\left({\frac{\alpha+1}{2}}\right), \\
    \int_0^\infty x^{2+\alpha} e^{-\frac{x^2}{2}}dx &= \frac{1}{2} 2^{\frac{\alpha+3}{2}}\Gamma\left({\frac{\alpha+3}{2}}\right),\\
    \int_0^\infty x^{4+\alpha} e^{-\frac{x^2}{2}}dx &= \frac{1}{2} 2^{\frac{\alpha+5}{2}}\Gamma\left({\frac{\alpha+5}{2}}\right).
\end{align*}
So we can compute both sides of the expression:
\begin{align*}
    RHS &= \frac{1}{2}\int_0^\infty |f'|^2 d\mu_\alpha = \frac{1}{2} \frac{\int_0^\infty 4x^{2+\alpha} e^{-\frac{x^2}{2}}dx}{\frac{1}{2} 2^{\frac{\alpha+1}{2}}\Gamma\left({\frac{\alpha+1}{2}}\right)} = \frac{1}{2} \cdot \frac{4 \frac{1}{2} 2^{\frac{\alpha+3}{2}}\Gamma\left({\frac{\alpha+3}{2}}\right)}{\frac{1}{2} 2^{\frac{\alpha+1}{2}}\Gamma\left({\frac{\alpha+1}{2}}\right)} = 4 \frac{\Gamma\left({\frac{\alpha+3}{2}}\right)}{\Gamma\left({\frac{\alpha+1}{2}}\right)} \\
    &= 2(\alpha + 1) \\
    LHS &= \operatorname{Var}_{\mu_\alpha}(f) = \int_0^\infty f^2 d\mu_\alpha - \left(\int_0^\infty f d\mu_\alpha\right)^2 = \frac{\int_0^\infty x^{4+\alpha}e^{-\frac{x^2}{2}}dx}{\int_0^\infty x^\alpha e^{-\frac{x^2}{2}}dx} - \left(\frac{\int_0^\infty x^{2+\alpha} e^{-\frac{x^2}{2}}dx}{\int_0^\infty x^\alpha e^{-\frac{x^2}{2}}dx}\right)^2 \\
    &= \frac{\frac{1}{2} 2^{\frac{\alpha+5}{2}}\Gamma\left({\frac{\alpha+5}{2}}\right)}{\frac{1}{2} 2^{\frac{\alpha+1}{2}}\Gamma\left({\frac{\alpha+1}{2}}\right)} - \left( \frac{\frac{1}{2} 2^{\frac{\alpha+3}{2}}\Gamma\left({\frac{\alpha+3}{2}}\right) }{\frac{1}{2} 2^{\frac{\alpha+1}{2}}\Gamma\left({\frac{\alpha+1}{2}}\right)}\right)^2 = 4 \left(\frac{\alpha+3}{2}\right)\left(\frac{\alpha+1}{2}\right) - \left(2 \left(\frac{\alpha+1}{2}\right)\right)^2 \\
    &=2 (\alpha + 1).
\end{align*}
So the function $f = x^2 $ is indeed the optimizer and so is $\varphi_1$.
\end{proof}
\begin{lemma}[Improved energy decay]\label{lem:decay}
For every $f\in  W^{1,2}_A(\mr_*^N, \mu_A)$ and every $t\geq 0$
\begin{equation}
    \E(P_tf)\ \leq\ e^{-4t}\,\E(f),\label{eq:decay-improved}
\end{equation}
where $\E(P_tf) = \int_{\mr_*^N} \Gamma(P_tf) d\mu_A = - \int_{\mr_*^N} P_tf\, L_A(P_tf) d\mu_A$.
\end{lemma}

\begin{proof}
We have that
\begin{align*}
    \frac{d}{dt}\,\E(P_tf)
    = -2\int_{\mr_*^N} |L_A P_tf|^2 d\mu_A
    = -2\int_{\mr_*^N} \Gamma_2(P_tf)\,d\mu_A.
\end{align*}
By the condition $ICD(2,\infty)$ in \eqref{ICD(2,infty)} we have
$$\int_{\mr_*^N}\Gamma_2(P_tf)\,d\mu_A\geq 2\int_{\mr_*^N}\Gamma(P_tf)\,d\mu_A=2\E(P_tf),$$
and thus
$$
\frac{d}{dt}\,\E(P_tf)\ \leq\ -4\,\E(P_tf).
$$
Hence $\frac{d}{dt}\left(e^{4t}\E(P_tf)\right)\leq0$, which gives \eqref{eq:decay-improved}. Alternatively, with the spectral decomposition of Section~\ref{sec:spectral} and $c_j=\|\Pi_jf\|_{L^2(\mu_A)}$, we have $\E(P_tf)=\sum_{j\geq1}2j\,e^{-4jt}c_j^2\leq e^{-4t}\sum_{j\geq1}2j\,c_j^2=e^{-4t}\E(f)$, which avoids any regularity issue.
\end{proof}

We combine this decay with the celebrated theorem of Gross \cite{GRO75}, see also \cite{GRO75b}, the survey \cite{Gross06}, and \cite{Gross25} for intrinsic hypercontractivity of Schr\"odinger operators.

\begin{theorem}[Gross' hypercontractivity]\label{thm:gross}
The logarithmic Sobolev inequality $LS(C,0)$ is equivalent to the following property: for all $1<p<q<\infty$ and all $t>0$ with $\frac{q -1}{p-1} \leq e^{2t/C}$,
\begin{equation}
    \|P_{t}f\|_{q} \leq\ \|f\|_p. \label{ineq:hypercontractivity}
\end{equation}
\end{theorem}

Since $L_A$ satisfies $CD(1,\infty)$, the measure $\mu_A$ satisfies $LS(1,0)$ (see Theorem~\ref{thm:B} and \cite{BE85}). By Gross' theorem this is equivalent to
hypercontractivity. Taking $q=2$ and $p\in(1,2)$, the condition reads
$e^{2t}\geq\frac1{p-1}$, i.e.\ $t\geq t_p=-\frac12\log(p-1)$.

\begin{theorem}[Improved Beckner inequality]\label{thm:beckner}
Assume $\alpha_i>0$ for all $i$. For every $p\in[1,2)$ and every $f\in  W^{1,2}_A(\mr_*^N, \mu_A)$
\begin{equation}
    \int_{\mr_*^N} |f|^2 d\mu_A
    -\left(\int_{\mr_*^N} |f|^p d\mu_A\right)^{2/p}
    \ \leq\ \frac{p(2-p)}{2}\,\E(f),\label{eq:beckner}
\end{equation}
equivalently
\begin{equation}
    \frac{\|f\|_2^2-\|f\|_p^2}{2-p}\ \leq\ \frac p2\,\E(f).\label{eq:beckner-norm}
\end{equation}
\end{theorem}

\begin{proof}
Fix $p\in(1,2)$. Using
$\frac{d}{ds}\|P_sf\|_2^2=-2\,\E(P_sf)$ together with Lemma \ref{lem:decay},
\begin{align*}
    \|f\|_2^2-\|P_tf\|_2^2
    &= 2\int_0^t \E(P_sf)\,ds
    \ \leq\ 2\E(f)\int_0^t e^{-4s}\,ds
    = \frac{1-e^{-4t}}{2}\,\E(f).
\end{align*}
Evaluate at $t=t_p=-\frac12\log(p-1)$. Then $e^{-2t_p}=p-1$, so
$e^{-4t_p}=(p-1)^2$ and
$$
\frac{1-e^{-4t_p}}{2}
= \frac{1-(p-1)^2}{2}
= \frac{\bigl(1-(p-1)\bigr)\bigl(1+(p-1)\bigr)}{2}
= \frac{p(2-p)}{2}.
$$
Finally, by the hypercontractive bound \eqref{ineq:hypercontractivity},
$\|P_{t_p}f\|_2^2\leq\|f\|_p^2$. Thus we get
$$
\|f\|_2^2-\|f\|_p^2
\ \leq\ \|f\|_2^2-\|P_{t_p}f\|_2^2
\ \leq\ \frac{p(2-p)}{2}\,\E(f).
$$
This proves the theorem for $p\in(1,2)$. The case $p=1$ follows by letting $p\downarrow1$.
\end{proof}

At $p=1$, \eqref{eq:beckner} gives the sharp Poincar\'e inequality of
Theorem~\ref{poincareineq} with constant $\frac12$. Indeed, for $f\geq0$ the left-hand side of \eqref{eq:beckner} with $p=1$ is $\operatorname{Var}_{\mu_A}(f)$. Since $\operatorname{Var}_{\mu_A}$ and $\E$ do not change when we add a constant, this gives the Poincar\'e inequality for bounded $f$, and then for all $f$ by density.

As $p\uparrow2$, the limit is
$$
\operatorname{Ent}_{\mu_A}(|f|^2)\leq\ 2\E(f),
$$
the sharp logarithmic Sobolev inequality, with the same constant as before. In other words, we do not get an improved constant for LSI. The constant is indeed sharp: for $f=e^{bx_1/2}$ we have $\log\int e^{bx_1}d\mu_A=\frac{b^2}{2}+\alpha_1\log b+O(1)$ as $b\to\infty$, and a direct computation gives $\operatorname{Ent}_{\mu_A}(|f|^2)/(2\E(f))\to1$. So the constant $\frac p2$ in \eqref{eq:beckner-norm} is sharp at both endpoints. For comparison, Theorem~\ref{thm:B} gives \eqref{eq:beckner-norm} with the constant $1$ in place of $\frac p2$, and that constant is sharp only for partial weights. For full weights, Theorem~\ref{thm:beckner} improves it for every $p\in[1,2)$. The interpolating family \eqref{eq:beckner} is due to Beckner \cite{Bec89}. Lei, Li, Wu, and Wu \cite{LLWW26} showed that the family is monotone: if the inequality $\frac{\|f\|_2^2-\|f\|_p^2}{2-p}\leq C\int|\nabla f|^2d\mu$ holds for all $f$, then it holds with the same $C$ for every exponent in $[1,p]$. Under the pointwise condition $CD(\rho,\infty)$ one has \eqref{eq:beckner-norm} with $\frac{1}{\rho}$ in place of $\frac{p}{2}$ (see \cite{BGL14, GZ21}). Our proof uses the integrated condition \eqref{ICD(2,infty)} for the energy
decay and the pointwise condition $CD(1,\infty)$ only through
hypercontractivity.

\begin{proof}[Proof of Theorem~\ref{thm:main}]
The condition $ICD(2,\infty)$ is Theorem~\ref{thm:ICD}. The Poincar\'e inequality and the characterization of its equality cases are Theorem~\ref{poincareineq}. By Section~\ref{sec:spectral}, the eigenspace of $-L_A$ with eigenvalue $2$ is $E_1=\operatorname{span}\{\Psi_k:|k|=1\}$, and for $|k|=1$ the function $\Psi_k$ is a constant multiple of $x_i^2-(\alpha_i+1)$ for some $i$. The Beckner inequality is \eqref{eq:beckner-norm} in Theorem~\ref{thm:beckner}, and its sharpness at $p=1$ and as $p\to2$ is shown after the proof of Theorem~\ref{thm:beckner}. Finally, let $q_1=x_1^2-(\alpha_1+1)$ and $f=1+\varepsilon q_1$ with $0<\varepsilon<\frac{1}{\alpha_1+1}$, so that $f>0$. Since $\int_{\mr_*^N}q_1d\mu_A=0$ and $q_1$ has finite moments of all orders, a Taylor expansion gives
\begin{equation*}
    \|f\|_2^2-\|f\|_p^2=(2-p)\,\varepsilon^2\int_{\mr_*^N}q_1^2d\mu_A+o(\varepsilon^2),\qquad \int_{\mr_*^N}|\nabla f|^2d\mu_A=2\varepsilon^2\int_{\mr_*^N}q_1^2d\mu_A.
\end{equation*}
Letting $\varepsilon\to0$ shows that the sharp constant in the Beckner inequality is at least $\frac12$ for every $p\in[1,2)$.
\end{proof}

\section{Spectral decomposition\texorpdfstring{---Proof of Theorem~\ref{poincare deficit}}{---Proof of Theorem 1.2}}\label{sec:spectral}
\subsection{The Laguerre operator}
In this subsection we compute the spectrum of $L_A$. Consider the one-dimensional operator
\begin{equation}
    L_{\alpha} = \p_{xx} - \left(x-\frac{\alpha}{x}\right)\p_x, \label{Laguerre}%
\end{equation}
where $\alpha > 0$. We would like to solve the eigenvalue equation
\begin{equation}
    L_\alpha f = f'' - \left(x-\frac{\alpha}{x}\right) f' = \lambda f. \label{eigenvalue}%
\end{equation}
This is closely related to the Laguerre operator
$$\mathcal{L}_\beta y = uy'' + (\beta+1-u) y'. $$
To make use of this, consider the change of variables $u = \frac{x^2}{2}$ and write $f(x)=g(u)$. Then $\frac{du}{dx} = x$ and $x^2 = 2u$, so
\begin{gather*}
f'(x)=x\,g'(u),\qquad f''(x) = x^2 g''(u) + g'(u) = 2u g''(u) + g'(u),\\
\left(x-\frac{\alpha}{x}\right)f'(x) = \left(x^2-\alpha\right)g'(u) = (2u - \alpha)g'(u).
\end{gather*}
Hence the equation \eqref{eigenvalue} becomes
\begin{align*}
    2u g'' + g' + ( \alpha - 2u)g'
    = 2ug'' + (\alpha +1 - 2u)g'
    = \lambda g.
\end{align*}
We can rewrite the resulting ODE as
\begin{equation}
    ug'' + \left(\frac{\alpha+1}{2}-u\right)g' = \frac{\lambda}{2}g. \label{ode}%
\end{equation}
Compare this to the generalized Laguerre equation
\begin{equation}
    uy'' + (\beta+1-u)y' = -ky.\label{gleq}
\end{equation}
We see that our ODE is exactly equation \eqref{gleq} with parameters $\beta = \frac{\alpha-1}{2}$ and $-k = \frac{\lambda}{2}$. For $k=0,1,2,\dots$ its polynomial solutions are the generalized Laguerre polynomials $L_k^\beta$:
\begin{align*}
    L_0^\beta (u) &= 1 \\
    L_1^\beta (u) &= 1+\beta - u\\
    L^\beta_2(u) &= \frac{1}{2}\left(u^2 -2(\beta+2)u + (\beta+1)(\beta+2)\right) \\
    \vdots
\end{align*}
Thus the eigenfunctions of \eqref{eigenvalue} obtained in this way are the functions $\psi_k(x)=L_k^\beta\left(\frac{x^2}{2}\right)$, with eigenvalues $\lambda=-2k$:
\begin{align*}
    \psi_0 &= 1 \\
    \psi_1 &= 1+\beta - u
    = 1+\frac{\alpha-1}{2} - \frac{x^2}{2} \\
    &= \frac{\alpha+1}{2}-\frac{x^2}{2} \\
    \psi_2 &= \frac{1}{2}\left(u^2 -2(\beta+2)u + (\beta+1)(\beta+2)\right) \\
    &= \frac{1}{2}\left(\frac{x^4}{4} -2\left(\frac{\alpha-1}{2}+2\right)\frac{x^2}{2} + \left(\frac{\alpha-1}{2}+1\right)\left(\frac{\alpha-1}{2}+2\right)\right) \\
    &= \frac{1}{8}\left(x^4 - 2(\alpha+3)x^2 +(\alpha+1)(\alpha+3) \right) \\
    \vdots
\end{align*}
\emph{Choice of the self-adjoint operator and completeness.} In the variable $u$, the point $u=0$ is of limit-circle type exactly when $\beta<1$, that is, when $\alpha<3$. In this range the second solution of \eqref{ode} behaves like $x^{1-\alpha}$ (or $\log x$ when $\alpha=1$) near $0$, and $L_\alpha$ has other self-adjoint realizations with a different spectrum. As fixed in the introduction, $L_\alpha$ is the operator associated with the form $\int_0^\infty|f'|^2d\mu_\alpha$ on the space $W^{1,2}_A$ of the introduction with $N=1$ (Neumann realization). The polynomials $\psi_k$ belong to its domain: for $h$ in this space, the boundary term $x^\alpha\psi_k'(x)h(x)$ vanishes at $x=0$ because $\psi_k'(x)=O(x)$. Since the Laguerre polynomials $\{L_k^\beta\}_{k\geq0}$ form a complete orthogonal system in $L^2(u^\beta e^{-u}du)$, the functions $\psi_k$ form a complete orthogonal system in $L^2(\mu_\alpha)$. Hence the spectrum of $-L_\alpha$ is exactly $\{2k:k\geq0\}$, and every eigenvalue is simple.

We extend this to several dimensions by tensorization. For a multi-index $k = (k_1, \dots, k_N)$ and $\beta_i = \frac{\alpha_i -1}{2}$ we define the multidimensional Laguerre polynomials
\begin{equation}
    \Psi_k (x) = \prod_{i=1}^N \psi_{k_i}(x_i) = \prod_{i=1}^N L_{k_i}^{\beta_i}\left(\frac{x_i^2}{2}\right).
\end{equation}
If $\alpha_i>0$ for every $i$, then $\Psi_k$ is an eigenfunction of $-L_A$:
$$-L_A\Psi_k = 2|k| \Psi_k.$$
Since tensor products of complete orthogonal systems are complete, $\{\Psi_k\}_k$ is a complete orthogonal system in $L^2(\mu_A)$. Therefore the spectrum of $-L_A$ is $\{2j:j\geq0\}$, and the eigenvalue $2j$ has the eigenspace $E_j=\operatorname{span}\{\Psi_k:|k|=j\}$, of dimension $\#\{k:|k|=j\}$ (which is larger than $1$ when $N\geq2$ and $j\geq1$). We denote by $\Pi_j$ the orthogonal projection in $L^2(\mu_A)$ onto $E_j$.
\subsection{The case of vanishing exponents}

In this subsection only, we allow some of the $\alpha_i$ to vanish. We split the coordinate indices
into the two sets
$$
\mathcal{Z} = \{\, i : \alpha_i = 0 \,\}, \qquad \mathcal{P} = \{\, i : \alpha_i > 0 \,\}
$$
on the domain $\mr_*^N = \{\, x \in \mr^N : x_i > 0 \ \text{whenever}\ \alpha_i > 0 \,\}$.

\subsection*{The one-dimensional Gaussian direction}
For $\alpha = 0$ the operator \eqref{Laguerre} becomes the standard
Ornstein--Uhlenbeck operator
\begin{equation}
    L_0 = \p_{xx} - x\,\p_x.   \label{OU}
\end{equation}
Its eigenfunctions are the Hermite polynomials $He_m$,
$$
He_0 = 1, \quad He_1 = x, \quad He_2 = x^2 - 1, \quad He_3 = x^3 - 3x, \quad \dots,
$$
satisfying $-L_0 He_m = m\,He_m$ for $m = 0,1,2,\dots$. Note that in the Gaussian case we get all the natural numbers as eigenvalues, not only the even
integers.

\subsection*{The general monomial}
Since $L_A = \sum_{i=1}^N L^{(i)}$ is a sum of one-dimensional
operators in independent variables, its eigenfunctions are products and its eigenvalues are sums. Each
$L^{(i)}$ is the Laguerre operator \eqref{Laguerre} when $i \in \mathcal{P}$ and the
Ornstein--Uhlenbeck operator \eqref{OU} when $i \in \mathcal{Z}$. Thus, for multi-indices
$m $ and $k $, define
\begin{equation}
    \Psi_{m,k}(x) = \prod_{i \in \mathcal{Z}} He_{m_i}(x_i)
    \prod_{i \in \mathcal{P}} L_{k_i}^{\beta_i}\!\left(\frac{x_i^2}{2}\right),
    \qquad \beta_i = \frac{\alpha_i - 1}{2}.
\end{equation}
Then $\Psi_{m,k}$ is an eigenfunction of $L_A$ with
$$
L_A \Psi_{m,k}
= -\left( \sum_{i \in \mathcal{Z}} m_i + 2\sum_{i \in \mathcal{P}} k_i \right) \Psi_{m,k},
$$
equivalently for the operator $-L_A$
$$
-L_A \Psi_{m,k}
= \left( \sum_{i \in \mathcal{Z}} m_i + 2\sum_{i \in \mathcal{P}} k_i \right) \Psi_{m,k}
= \bigl( |m| + 2|k| \bigr)\Psi_{m,k}.
$$
The $\mathcal{Z}$-directions contribute to the eigenvalue in unit steps and the
$\mathcal{P}$-directions in steps of two. Setting $\mathcal{P} = \varnothing$ (that is, $A = 0$)
recovers the Gaussian case $-L_A \Psi_m = |m|\,\Psi_m$, while $\mathcal{Z} = \varnothing$
recovers the purely weighted case $-L_A \Psi_k = 2|k|\,\Psi_k$.

\subsection*{Verification for the mixed case}
Take $N = 2$ with $\alpha_1 = 0$ and $\alpha_2 = \alpha > 0$, so
$$
L_A = \p_{x_1 x_1} - x_1 \p_{x_1}
    + \p_{x_2 x_2} - \left(x_2 - \frac{\alpha}{x_2}\right)\p_{x_2}.
$$
Consider the eigenfunction indexed by $m = 1$, $k = 1$, namely
$$
\Phi = He_1(x_1)\,(2\psi_1)(x_2) = x_1\bigl[(\alpha+1) - x_2^2\bigr]
$$
for which we expect the eigenvalue $-(m + 2k) = -3$. Computing term by term,
\begin{align*}
    \p_{x_1 x_1}\Phi &= 0 \\
    - x_1 \p_{x_1}\Phi &= -x_1\bigl[(\alpha+1) - x_2^2\bigr] \\
    \p_{x_2 x_2}\Phi &= -2x_1 \\
    -\left(x_2 - \frac{\alpha}{x_2}\right)\p_{x_2}\Phi
    &= -\left(x_2 - \frac{\alpha}{x_2}\right)(-2x_1 x_2)
    = 2x_1 x_2^2 - 2\alpha x_1.
\end{align*}
Adding these and simplifying we get:
\begin{align*}
    L_A \Phi
    &= -(\alpha+1)x_1 + x_1 x_2^2 - 2x_1 + 2x_1 x_2^2 - 2\alpha x_1 \\
    &= -3(\alpha+1)x_1 + 3x_1 x_2^2 \\
    &= -3\bigl[(\alpha+1)x_1 - x_1 x_2^2\bigr] = -3\,\Phi.
\end{align*}
Thus $\Phi$
is an eigenfunction with $L_A \Phi = -3\Phi$, that is, $-L_A \Phi = 3\Phi
= (m + 2k)\Phi$. Another eigenfunction with the same eigenvalue is $\Phi = x_1^3-3x_1$, which corresponds to $m=3$, $k=0$.
\begin{align*}
    \p_{x_1 x_1}\Phi &= 6x_1 \\
    - x_1 \p_{x_1}\Phi &= -3x_1^3+3x_1 \\
    \p_{x_2 x_2}\Phi &= 0 \\
    -\left(x_2 - \frac{\alpha}{x_2}\right)\p_{x_2}\Phi
    &= 0.
\end{align*}
Adding these and simplifying we get:
\begin{align*}
    L_A \Phi
    &= 6x_1 -3x_1^3 +3x_1  \\
    &= -3(x_1^3 -3x_1) \\
    &= -3\Phi.
\end{align*}

\subsection{Improved Poincar\'e inequality}
In this subsection we show an improvement of the Poincar\'e inequality and prove Theorem~\ref{poincare deficit}.
\begin{theorem}\label{improved poincare}
Assume $\alpha_i > 0$ for all $i$. Define $\bar q = (\bar q_1, \dots, \bar q_N)$ by
$$\bar q_i = \frac{x_i^2 - (\alpha_i+1)}{\sqrt{2(\alpha_i+1)}}.$$
Then for all $u \in  W^{1,2}_A(\mr_*^N, \mu_A)$
\begin{gather}
    \int_{\mr_*^N} |\nabla u|^2  d\mu_A - 2\int_{\mr_*^N} \left|u - \int_{\mr_*^N} u  d\mu_A\right|^2 d\mu_A \notag \\ \geq \frac{1}{2}\int_{\mr_*^N} \left|\nabla\left(u - \int_{\mr_*^N} u  d\mu_A - \left(\int_{\mr_*^N} \bar q  u  d\mu_A\right)\cdot \bar q\right)\right|^2 d\mu_A. \label{IP}%
\end{gather}
Equality holds if and only if $u=\Pi_0u+\Pi_1u+\Pi_2u$.
\end{theorem}
\begin{proof}
Write $u=\sum_{j\geq0}\Pi_ju$ and $c_j=\|\Pi_ju\|_{L^2(\mu_A)}$. Since the eigenspaces $E_j$ are orthogonal, $-L_A\Pi_ju=2j\Pi_ju$, and $W^{1,2}_A(\mr_*^N,\mu_A)$ is the form domain of $-L_A$, the spectral theorem gives
$$\int_{\mr_*^N} |u|^2 d\mu_A = \sum_{j\geq 0} c_j^2, \qquad \int_{\mr_*^N} |\nabla u|^2 d\mu_A = \sum_{j\geq 0} 2j  c_j^2.$$
Without loss of generality, assume $\int_{\mr_*^N} u  d\mu_A = 0$, so that $c_0 = 0$, since subtracting a constant changes neither side. The left-hand side is then
\begin{equation}
    \int_{\mr_*^N} |\nabla u|^2 d\mu_A - 2\int_{\mr_*^N} |u|^2 d\mu_A = \sum_{j\geq 1} (2j - 2) c_j^2 = \sum_{j\geq 2} (2j-2) c_j^2, \label{eq:spdeficit}
\end{equation}
where the eigenvalue-$2$ terms ($j=1$) vanish since $2j - 2 = 0$. The eigenspace of eigenvalue $2$ is spanned by the quadratics $q_i = x_i^2 - (\alpha_i+1)$ with $\int_{\mr_*^N} q_i^2  d\mu_A = 2(\alpha_i+1)$. Normalizing, $\bar q_i = q_i/\sqrt{2(\alpha_i+1)}$, the projection $\Pi_1 u$ is
\begin{equation}
    \Pi_1 u = \left(\int_{\mr_*^N} \bar q  u  d\mu_A\right)\cdot \bar q = \sum_{i=1}^N \left(\int_{\mr_*^N} \bar q_i  u  d\mu_A\right)\bar q_i.
\end{equation}
We have $ u - \Pi_1 u =\sum_{j \geq 2} \Pi_ju$ and
\begin{equation}
    \int_{\mr_*^N} |\nabla(u-\Pi_1u)|^2 d\mu_A = \sum_{j \geq 2} 2j c_j^2. \label{eq:firstremainder}
\end{equation}
For $j \geq 2$ the eigenvalue $\lambda_j=2j$ satisfies $\lambda_j \geq 4$. Thus
$$\lambda_j - 2 \geq \frac{1}{2}\lambda_j \iff \frac{1}{2}\lambda_j \geq 2 \iff \lambda_j \geq 4,$$
so applying this to \eqref{eq:spdeficit} we get
$$\sum_{j\geq 2} (2j-2) c_j^2 \geq \frac{1}{2}\sum_{j\geq 2} 2j  c_j^2 = \frac{1}{2}\int_{\mr_*^N} |\nabla (u - \Pi_1u)|^2 d\mu_A.$$

\begin{gather}
    \int_{\mr_*^N} |\nabla u|^2 d\mu_A - 2\int_{\mr_*^N} \left|u - \int_{\mr_*^N} u  d\mu_A\right|^2 d\mu_A \notag \\
    \geq \frac{1}{2}\int_{\mr_*^N} \left|\nabla\left(u - \int_{\mr_*^N} u  d\mu_A - \left(\int_{\mr_*^N} \bar q  u  d\mu_A\right)\cdot \bar q\right)\right|^2 d\mu_A. \label{ineq:imp Poincare}%
\end{gather}
The only inequality used is $2j - 2 \geq j$ for $j \geq 2$, an equality at $j = 2$. Hence equality holds if and only if $c_j = 0$ for all $j > 2$, that is $u=\Pi_0u+\Pi_1u+\Pi_2u$.
\end{proof}
\begin{corollary}\label{cor:L2 poincare}
    Assume $\alpha_j > 0$ for all $j$ and let $\bar{q}$ be as in Theorem~\ref{improved poincare}. Then for all $u\in W^{1,2}_A(\mr_*^N,\mu_A)$,
    \begin{multline*}
        \int_{\mr_*^N} |\nabla u|^2 d\mu_A - 2\int_{\mr_*^N} \left|u - \int_{\mr_*^N} u  d\mu_A\right|^2 d\mu_A
        \\ \geq 2\int_{\mr_*^N} \left| u -\int_{\mr_*^N} u  d\mu_A - \left(\int_{\mr_*^N} \bar q  u  d\mu_A\right) \cdot \bar{q} \right|^2 d\mu_A.
    \end{multline*}
    The constant $2$ is sharp, and equality holds if and only if $u=\Pi_0u+\Pi_1u+\Pi_2u$.
\end{corollary}
\begin{proof}
    As in the proof of Theorem~\ref{improved poincare} we may assume $\int_{\mr_*^N}u\,d\mu_A=0$. Then $u -\left(\int_{\mr_*^N} \bar q  u  d\mu_A\right) \cdot \bar{q}=u-\Pi_1u=\sum_{j\geq2}\Pi_ju$, so its squared norm is $\sum_{j\geq2}c_j^2$. By \eqref{eq:spdeficit},
    $$\int_{\mr_*^N} |\nabla u|^2 d\mu_A - 2\int_{\mr_*^N} |u|^2 d\mu_A=\sum_{j\geq2}(2j-2)c_j^2\geq 2\sum_{j\geq2}c_j^2,$$
    since $2j-2\geq2$ for $j\geq2$, with equality only for $j=2$.
\end{proof}

\begin{proof}[Proof of Theorem~\ref{poincare deficit}]
    We may again assume $\int_{\mr_*^N}u\,d\mu_A=0$. By definition of $R_iu$ we have
    \begin{equation}
        \int_{\mr_*^N} \left|\nabla R_iu\right|^2 d\mu_A  =\sum_{j=i+1}^\infty 2j\int_{\mr_*^N}|\Pi_ju|^2d\mu_A = \sum_{j =i+1}^\infty 2jc_j^2.\label{eq:i-th remainder}
    \end{equation}
From \eqref{eq:spdeficit} we have the identity
    \begin{align*}
        \int_{\mr_*^N} |\nabla u|^2 d\mu_A - 2\int_{\mr_*^N} |u|^2 d\mu_A &=  \sum_{i\geq 2} \frac{2}{1}(i-1) c_i^2 \\
        &= \frac{1}{2}\sum_{i\geq 2}2i c_i^2 + \sum_{i\geq 2}\frac{2}{2}(i-2) c_i^2 \\
      \text{by \eqref{eq:firstremainder}}  &= \frac{1}{2}\int_{\mr_*^N} |\nabla R_1u|^2 d\mu_A + \sum_{i\geq 3}(i-2) c_i^2 \\
      &= \frac{1}{2}\int_{\mr_*^N} |\nabla R_1u|^2 d\mu_A + \sum_{i\geq 3}\left(\frac{i}{3}+\frac{2(i-3)}{3}\right) c_i^2 \\
      \text{by \eqref{eq:i-th remainder}} &= \frac{1}{2}\int_{\mr_*^N} |\nabla R_1u|^2 d\mu_A + \frac{1}{6} \int_{\mr_*^N} |\nabla R_2u|^2 d\mu_A +  \sum_{i\geq 4}\frac{2(i-3)}{3} c_i^2 \\
      \vdots
    \end{align*}
    In general, after $\ell-1$ steps the factor in front of $c_i^2$ is $r_\ell = \frac{2}{\ell}(i-\ell)$. When we extract the next remainder, we have the relation
    \begin{equation}
        \frac{2}{\ell}(i-\ell) = 2i b_\ell + \frac{2}{\ell+1}(i-(\ell+1)).\label{eq:coefficients}
    \end{equation}
    Solving for $b_\ell$ we get
    \begin{equation}
        b_\ell = \frac{1}{\ell(\ell+1)}.
    \end{equation}
    Thus we get
    \begin{align}
         &\int_{\mr_*^N} |\nabla u|^2 d\mu_A - 2\int_{\mr_*^N} |u|^2 d\mu_A \notag\\
         =& \frac{1}{2}\int_{\mr_*^N} |\nabla R_1u|^2 d\mu_A + \dots +  b_\ell\int_{\mr_*^N} |\nabla R_\ell u|^2 d\mu_A  +\sum_{i\geq \ell+2}\frac{2}{\ell+1}(i-(\ell+1)) c_i^2 \notag\\
         =& \sum_{i=1}^\infty \frac{1}{i(i+1)} \int_{\mr_*^N} |\nabla R_iu|^2 d\mu_A.
    \end{align}
    In the last step we used that the remainder tends to $0$:
    $$\sum_{i\geq \ell+2}\frac{2}{\ell+1}(i-(\ell+1)) c_i^2\leq\frac{1}{\ell+1}\sum_{i\geq1}2ic_i^2=\frac{1}{\ell+1}\int_{\mr_*^N}|\nabla u|^2d\mu_A\longrightarrow0\qquad(\ell\to\infty).$$
\end{proof}
\section{Stability of the Brascamp--Lieb inequality\texorpdfstring{---Proofs of Theorems~\ref{thm:BL stability} and \ref{thm:BL stability homogeneous}}{---Proofs of Theorems 1.3 and 1.4}}\label{sec:stability of BL}
\subsection{Monomial case}
In this subsection we prove the stability of the Brascamp--Lieb inequality in the monomial case.
\begin{proof}[Proof of Theorem~\ref{thm:BL stability}]
    Since $\int_{\mr_*^N}f\,d\mu_A=0$ and $\mu_A$ satisfies the Poincar\'e inequality, there is a unique solution $u$ of the Poisson equation $-L_Au = f$ with $\int_{\mr_*^N}u\,d\mu_A=0$. Integrating by parts (the boundary terms vanish since $x_i^{\alpha_i}=0$ on $\{x_i=0\}$), $\int_{\mr_*^N} f x_i\,d\mu_A=\int_{\mr_*^N}(-L_Au)x_i\,d\mu_A=\int_{\mr_*^N}\nabla u\cdot\nabla x_i\,d\mu_A=\int_{\mr_*^N}\p_iu\,d\mu_A$. Hence
    \begin{align}
        \int_{\mr_*^N} \nabla u d\mu_A = \int_{\mr_*^N} fx d\mu_A = a.
    \end{align}
    Next we define $h = u - a\cdot x$ and it follows that $$v = \nabla h  = \nabla u - a,$$
    and also note that
    \begin{equation}
        \int_{\mr_*^N} v d\mu_A = \int_{\mr_*^N} \nabla u d\mu_A - a = 0.
    \end{equation}
    Next we write the deficit in the dual form. Let $B=\nabla^2W$. Differentiating $L_Au=\Delta u-\nabla W\cdot\nabla u$ gives the commutation formula $\nabla L_Au=L_A\nabla u-B\nabla u$, where $L_A$ acts on each component. Hence $\nabla f=(-L_A)\nabla u+B\nabla u$ and
    \begin{equation*}
        \langle B^{-1}\nabla f,\nabla f\rangle=\langle B^{-1}(-L_A)\nabla u,(-L_A)\nabla u\rangle+2\langle(-L_A)\nabla u,\nabla u\rangle+\langle B\nabla u,\nabla u\rangle.
    \end{equation*}
    Moreover $\int_{\mr_*^N}\langle(-L_A)\nabla u,\nabla u\rangle d\mu_A=\int_{\mr_*^N}\|\nabla^2u\|^2d\mu_A$ and
    \begin{equation*}
    \operatorname{Var}_{\mu_A}(f)=\int_{\mr_*^N}|L_Au|^2d\mu_A=\int_{\mr_*^N}\Gamma_2(u)\,d\mu_A=\int_{\mr_*^N}\left(\|\nabla^2u\|^2+\langle B\nabla u,\nabla u\rangle\right)d\mu_A.
    \end{equation*}
    Integrating the previous identity, we obtain (the integrations by parts are justified by the cutoff argument of Lemma~\ref{lemma:IBPw identity}, since the weight vanishes on $\p\mr_*^N$)
    \begin{equation}
        \delta_{BL}(f) = \int_{\mr_*^N} \|\nabla^2 u \|^2 d\mu_A + \int_{\mr_*^N} \langle B^{-1}(-L_A)\nabla u, (-L_A)\nabla u\rangle d\mu_A. \label{eq:deficit in dr}%
    \end{equation}
    Since $(-L_A)(a) = 0$ and $\nabla v = \nabla^2 u$ we can rewrite \eqref{eq:deficit in dr} in terms of $v$:
    \begin{equation}
        \delta_{BL}(f) = \underbrace{\int_{\mr_*^N} \|\nabla v \|^2 d\mu_A}_{\mathcal A} + \underbrace{\int_{\mr_*^N} \langle B^{-1}(-L_A) v, (-L_A) v\rangle d\mu_A }_{\mathcal C}.\label{eq:deficit in dr in v}%
    \end{equation}
    Also note that
    \begin{align}
        -L_Ah &= -L_A(u - a\cdot x) \notag \\
        &= f +L_A(a\cdot x) \notag \\
        &= f +(\Delta - \nabla W \cdot \nabla)(a\cdot x) \notag \\
        &= f - \langle\nabla W, a\rangle. \label{eq:Lh}%
    \end{align}
    So we get
    \begin{align}
        \int_{\mr_*^N} \left|f - \langle\nabla W, a\rangle\right|^2 d\mu_A &= \int_{\mr_*^N} |L_Ah|^2 d\mu_A \notag \\
        &= \int_{\mr_*^N} \Gamma_2(h)  d\mu_A \notag \\
        &= \int_{\mr_*^N} \|\nabla^2 h\|^2  + \langle \nabla^2W \nabla h, \nabla h\rangle d\mu_A \notag \\
        &= \underbrace{\int_{\mr_*^N} \|\nabla v\|^2 d\mu_A}_{\mathcal A} + \underbrace{\int_{\mr_*^N} \langle B  v,  v\rangle d\mu_A }_{\mathcal D}. \label{BLstab:distance to opt}
    \end{align}
    Since $\int_{\mr_*^N}v_i\,d\mu_A=0$, the Poincar\'e inequality of Theorem~\ref{poincareineq} applied to each component $v_i$ gives
    \begin{equation}
        \int_{\mr_*^N} |v_i|^2 d\mu_A \leq \frac{1}{2} \int_{\mr_*^N} |\nabla v_i|^2d\mu_A. \label{BLstab:helper1}
    \end{equation}
    Note that
    \begin{equation}
        (\alpha_i-1)x_i^{\alpha_i-2}e^{-\frac{x_i^2}{2}} = \p_i \left( x_i^{\alpha_i-1}e^{-\frac{x_i^2}{2}}\right) + x_i^{\alpha_i}e^{-\frac{x_i^2}{2}}, \label{BLstab:helper2}
    \end{equation}
    and that $\int_{\mr_*^N}\frac{|v_i|^2}{x_i^2}d\mu_A<\infty$. Indeed, by Lemma~\ref{lem:monomial hypotheses} the weight $w=x^A$ satisfies the hypotheses of Lemma~\ref{lemma:IBPw identity} with $\theta=\max_i\frac1{\alpha_i}<1$. Moreover $v\in L^2(\mu_A)$, and $\int_{\mr_*^N}\|\nabla v\|^2d\mu_A\leq\int_{\mr_*^N}\Gamma_2(u)\,d\mu_A=\int_{\mr_*^N}|f|^2d\mu_A<\infty$. Hence $\int_{\mr_*^N}\sum_j\frac{\alpha_j^2}{x_j^2}|v|^2d\mu_A<\infty$. This finiteness is needed to solve the quadratic inequality \eqref{BLstab:helper4} below. The boundary terms at $x_i=0$ and $x_i=\infty$ in the integration by parts below vanish. This is justified with the cutoff functions exactly as in the proof of Lemma~\ref{lemma:IBPw identity}.
    Thus, using Tonelli's theorem, integration by parts in $x_i$ and \eqref{BLstab:helper2} (up to the normalizing constant of $\mu_{\alpha_i}$, which appears on both sides), we get
    \begin{align}
       \int_{\mr_*^N}(\alpha_i-1) \frac{|v_i|^2}{x_i^2} d\mu_A &= \int_{\mr_*^{N-1}} \int_{0}^\infty (\alpha_i-1) \frac{|v_i|^2}{x_i^2} d\mu_{\alpha_i} d\mu_{\neq i} \notag\\
       &= \int_{\mr_*^{N-1}} \int_{0}^\infty (\alpha_i-1) |v_i|^2 x_i^{\alpha_i-2} e^{-\frac{x_i^2}{2}}dx_i d\mu_{\neq i} \notag\\
       &\overset{\eqref{BLstab:helper2}}{=} \int_{\mr_*^{N-1}} \int_{0}^\infty  |v_i|^2 \left(\p_i \left(x_i^{\alpha_i-1}e^{-\frac{x_i^2}{2}}\right) + x_i^{\alpha_i}e^{-\frac{x_i^2}{2}}\right)dx_i d\mu_{\neq i} \notag \\
       &= \int_{\mr_*^{N-1}} \int_{0}^\infty  |v_i|^2 \p_i \left(x_i^{\alpha_i-1}e^{-\frac{x_i^2}{2}}\right) dx_i d\mu_{\neq i} \notag \\ &\qquad + \int_{\mr_*^{N-1}} \int_{0}^\infty  |v_i|^2 x_i^{\alpha_i}e^{-\frac{x_i^2}{2}}dx_i d\mu_{\neq i} \notag \\
       &= \int_{\mr_*^{N-1}} \int_{0}^\infty -\p_i(|v_i|^2) \left(x_i^{\alpha_i-1}e^{-\frac{x_i^2}{2}}\right) dx_i d\mu_{\neq i} + \int_{\mr_*^{N}}  |v_i|^2 d\mu_A \notag \\
       &=  -2\int_{\mr_*^{N-1}} \int_{0}^\infty v_i \p_iv_i\, x_i^{\alpha_i-1}e^{-\frac{x_i^2}{2}} dx_i d\mu_{\neq i} + \int_{\mr_*^{N}}  |v_i|^2 d\mu_A \notag \\
       &= -2\int_{\mr_*^{N}} \frac{v_i \p_iv_i}{x_i} d\mu_A + \int_{\mr_*^{N}}  |v_i|^2 d\mu_A \notag \\
       &\leq 2 \left|\int_{\mr_*^{N}} \frac{v_i \p_iv_i}{x_i} d\mu_A\right| + \int_{\mr_*^{N}}  |v_i|^2 d\mu_A \notag \\
       &\overset{\text{C-S}}{\leq} 2 \left(\int_{\mr_*^{N}} \frac{|v_i|^2}{x_i^2} d\mu_A\right)^{\frac{1}{2}}\left(\int_{\mr_*^{N}} |\p_i v_i|^2d\mu_A\right)^{\frac{1}{2}} + \int_{\mr_*^{N}}  |v_i|^2 d\mu_A \notag \\
       & \overset{\eqref{BLstab:helper1}}{\leq}2 \left(\int_{\mr_*^{N}} \frac{|v_i|^2}{x_i^2} d\mu_A\right)^{\frac{1}{2}}\left(\int_{\mr_*^{N}} |\nabla v_i|^2 d\mu_A\right)^{\frac{1}{2}} + \frac{1}{2}\int_{\mr_*^{N}}  |\nabla v_i|^2 d\mu_A. \label{BLstab:helper3}
    \end{align}
    If $\int_{\mr_*^{N}}  |\nabla v_i| ^2 d\mu_A=0$, then $v_i$ is constant, hence $v_i=0$, and there is nothing to prove. Otherwise, dividing \eqref{BLstab:helper3} by $\int_{\mr_*^{N}}  |\nabla v_i| ^2 d\mu_A $ we get
    \begin{equation*}
        (\alpha_i-1)\frac{\int_{\mr_*^N}\frac{|v_i|^2}{x_i^2} d\mu_A }{\int_{\mr_*^{N}} |\nabla v_i|^2 d\mu_A } \leq 2 \frac{\left(\int_{\mr_*^{N}} \frac{|v_i|^2}{x_i^2} d\mu_A\right)^{\frac{1}{2}}}{\left(\int_{\mr_*^{N}} |\nabla v_i|^2 d\mu_A\right)^{\frac{1}{2}}} + \frac{1}{2}.
    \end{equation*}
    Let $y_i = \frac{\left(\int_{\mr_*^{N}} \frac{|v_i|^2}{x_i^2} d\mu_A\right)^{\frac{1}{2}}}{\left(\int_{\mr_*^{N}} |\nabla v_i|^2 d\mu_A\right)^{\frac{1}{2}}} $, which is finite. Then we can rewrite this as
    \begin{equation}
        (\alpha_i-1)y_i^2 -2y_i -\frac{1}{2} \leq 0. \label{BLstab:helper4}
    \end{equation}
    Solving \eqref{BLstab:helper4} we get
    \begin{equation}
        y_i \leq \frac{1+\sqrt{\frac{\alpha_i+1}{2}}}{\alpha_i-1}.\label{BLstab:helper5}
    \end{equation}
    Thus
    \begin{align}
         \frac{\int_{\mr_*^{N}} \frac{|v_i|^2}{x_i^2} d\mu_A}{\int_{\mr_*^{N}} |\nabla v_i|^2 d\mu_A} =  y_i^2 &\leq \left(\frac{1+\sqrt{\frac{\alpha_i+1}{2}}}{\alpha_i-1}\right)^2 \notag\\
         \int_{\mr_*^{N}} \frac{|v_i|^2}{x_i^2} d\mu_A &\leq \left(\frac{1+\sqrt{\frac{\alpha_i+1}{2}}}{\alpha_i-1}\right)^2\int_{\mr_*^{N}} |\nabla v_i|^2 d\mu_A. \label{BLstab:helper6}
    \end{align}
    Combining \eqref{BLstab:helper1} and \eqref{BLstab:helper6} we get the estimate
    \begin{align}
        \int_{\mr_*^{N}} \left(1+\frac{\alpha_i}{x_i^2}\right) |v_i|^2  d\mu_A &\leq \frac{1}{2} \int_{\mr_*^N} |\nabla v_i|^2d\mu_A  +  \alpha_i\left(\frac{1+\sqrt{\frac{\alpha_i+1}{2}}}{\alpha_i-1}\right)^2\int_{\mr_*^{N}} |\nabla v_i|^2 d\mu_A \notag\\
        &=\left(\frac{1}{2} + \alpha_i\left(\frac{1+\sqrt{\frac{\alpha_i+1}{2}}}{\alpha_i-1}\right)^2\right)\int_{\mr_*^{N}} |\nabla v_i|^2 d\mu_A. \label{BLstab:helper7}
    \end{align}
    Summing \eqref{BLstab:helper7} over $i$ we get
    \begin{align}
        \underbrace{\int_{\mr_*^{N}}  \sum_i \left(1+\frac{\alpha_i}{x_i^2}\right) |v_i|^2}_{\mathcal D}  d\mu_A &\leq \sum_i \left(\frac{1}{2} + \alpha_i\left(\frac{1+\sqrt{\frac{\alpha_i+1}{2}}}{\alpha_i-1}\right)^2\right)\int_{\mr_*^{N}} |\nabla v_i|^2 d\mu_A \notag \\
        &\leq \max_{i}\left[\frac{1}{2} + \alpha_i\left(\frac{1+\sqrt{\frac{\alpha_i+1}{2}}}{\alpha_i-1}\right)^2\right] \sum_i\int_{\mr_*^{N}} |\nabla v_i|^2 d\mu_A \notag \\
        &= \kappa_A^* \int_{\mr_*^{N}} \sum_i\sum_j |\p_jv_i|^2 d\mu_A \notag \\
        &= \kappa_A^* \underbrace{\int_{\mr_*^{N}} \|\nabla v\|^2 d\mu_A }_{\mathcal A}. \label{BLstab:ADbound}
    \end{align}
    Here $\kappa_A^* = \max_{i}\left[\frac{1}{2} + \alpha_i\left(\frac{1+\sqrt{\frac{\alpha_i+1}{2}}}{\alpha_i-1}\right)^2\right]=\frac{1}{\kappa_A}$. Next consider the term
    \begin{align}
        \mathcal A &=  \int_{\mr_*^{N}}\|\nabla v\|^2 d\mu_A  =  \int_{\mr_*^{N}} \langle -L_Av, v\rangle d\mu_A \notag\\
        &=  \int_{\mr_*^{N}} \langle B^{\frac{1}{2}}B^{-\frac{1}{2}}(-L_Av), v\rangle d\mu_A \notag \\
        &= \int_{\mr_*^{N}} \langle B^{-\frac{1}{2}}(-L_Av), B^{\frac{1}{2}}v\rangle d\mu_A \notag \\
       \text{by C-S} \quad &\leq \left(\int_{\mr_*^{N}} \langle B^{-\frac{1}{2}}(-L_Av), B^{-\frac{1}{2}}(-L_Av)\rangle d\mu_A\right)^{\frac{1}{2}}\left(\int_{\mr_*^{N}} \langle B^{\frac{1}{2}}v, B^{\frac{1}{2}}v\rangle d\mu_A\right)^{\frac{1}{2}} \notag \\
       &= \left(\int_{\mr_*^{N}} \langle B^{-1}(-L_Av), (-L_Av)\rangle d\mu_A\right)^{\frac{1}{2}}\left(\int_{\mr_*^{N}} \langle Bv, v\rangle d\mu_A\right)^{\frac{1}{2}}  \notag \\
       &= \mathcal C^{\frac{1}{2}}\mathcal D^{\frac{1}{2}}.\label{BLstab:helper8}
    \end{align}
In other words, \eqref{BLstab:helper8} gives us the bound (if $\mathcal D=0$ then $v=0$ and there is nothing to prove)
\begin{equation}
    \mathcal A^2 \leq \mathcal C\mathcal D \iff \mathcal C \geq \frac{\mathcal A^2}{\mathcal D} \implies \mathcal C\geq \frac{1}{\kappa_A^*}\mathcal A \quad \text{by \eqref{BLstab:ADbound}} \label{BLstab:ACbound}
\end{equation}
Finally putting everything together we have

    \begin{align}
        \delta_{BL}(f) &= \mathcal A+\mathcal C \qquad \text{by \eqref{eq:deficit in dr in v}} \notag \\
        &\geq  \mathcal A +  \frac{1}{\kappa_A^*}\mathcal A\qquad \text{by \eqref{BLstab:ACbound}} \notag \\
        &\geq  \frac{1}{\kappa_A^*}\mathcal D +  \frac{1}{\kappa_A^*}\mathcal A\qquad \text{by \eqref{BLstab:ADbound}} \notag \\
        &=  \frac{1}{\kappa_A^*} \int_{\mr_*^N} \left|f - \langle\nabla W, a\rangle\right|^2 d\mu_A \qquad \text{by \eqref{BLstab:distance to opt}} \label{BLstab:specific a} \\
        &\geq \frac{1}{\kappa_A^*} \inf_{c\in\mr^N} \int_{\mr_*^N} \left|f - \langle\nabla W, c\rangle\right|^2 d\mu_A. \label{BLstability}%
    \end{align}
    This gives the $L^2$ stability.
\end{proof}
\begin{corollary}[Gradient stability of BL] \label{thm:gradient stability}
    Under the same assumptions as in Theorem~\ref{thm:BL stability}, we have
    \begin{equation}
        \delta_{BL}(f) \geq C_A \int_{\mr_*^{N}} \langle B^{-1}(\nabla f - Ba), \nabla f - Ba \rangle d\mu_A, \label{eq:BL gradient stability}
    \end{equation}
    where $B=\nabla^2W$, $a=\int_{\mr_*^N}xf\,d\mu_A$, and $C_A=\frac{\kappa_A}{1+\kappa_A}$.
\end{corollary}
\begin{proof}
    Let $g = f - \langle \nabla W, a\rangle$. Then
    $$\nabla g = \nabla f - Ba,$$
    and applying the deficit to $g$ we get
    \begin{align*}
        \delta_{BL}(g) &= \int_{\mr_*^N} \langle B^{-1}\nabla g, \nabla g\rangle d\mu_A - \var(g) \\
        &= \int_{\mr_*^N} \langle B^{-1}\nabla g, \nabla g\rangle d\mu_A - \int_{\mr_*^N} |g|^2 d\mu_A.
    \end{align*}
    Note that $g$ may fail to be in $W^{1,2}_A(\mr_*^N,\mu_A)$ when some $\alpha_i\leq3$, but $\delta_{BL}(g)$ is finite since $\langle B^{-1}Ba,Ba\rangle=\sum_ia_i^2\left(1+\frac{\alpha_i}{x_i^2}\right)$ is integrable for $\alpha_i>1$. Here $\int_{\mr_*^N}g\,d\mu_A=0$, since $\int_{\mr_*^N}\p_iW\,d\mu_A=0$ for every $i$ (by integration by parts, where the boundary terms vanish because $\alpha_i>1$). Moreover $\delta_{BL}$ is a nonnegative quadratic form which vanishes at $\langle\nabla W,a\rangle$ (the equality case of \eqref{ineq:Brascamp Lieb}). By the Cauchy--Schwarz inequality for the associated bilinear form, $\langle\nabla W,a\rangle$ lies in its kernel. We therefore get $\delta_{BL}(f) = \delta_{BL}(g)$, so
    \begin{align}
        \int_{\mr_*^N} \langle B^{-1}\nabla g, \nabla g\rangle d\mu_A &= \delta_{BL}(f) + \int_{\mr_*^N} |g|^2 d\mu_A \notag \\
        \text{by \eqref{BLstab:specific a}} \qquad &\leq \delta_{BL}(f) + \frac{1}{\kappa_A}\delta_{BL}(f) \notag \\
        &= \frac{\kappa_A + 1}{\kappa_A} \delta_{BL}(f).
    \end{align}
    Note that \eqref{eq:BL Stability}, which has the infimum over $c$, is not enough here. We need the bound $\int |g|^2d\mu_A\leq\frac{1}{\kappa_A}\delta_{BL}(f)$ for the specific vector $a$, and this is \eqref{BLstab:specific a}. This gives \eqref{eq:BL gradient stability} with $C_A=\frac{\kappa_A}{1+\kappa_A}$.
\end{proof}
\subsection{Homogeneous case}
In this subsection we extend the proof to homogeneous Gaussian measures under an additional subharmonicity condition. We begin with a technical lemma.

\begin{lemma}[Weighted IBP identity]\label{lemma:IBPw identity}
    Let $w$ be homogeneous of degree $\alpha>0$ and log-concave on the open convex cone
    $\Sigma=\{w>0\}$, with $w\in C(\overline\Sigma)\cap C^2(\Sigma)$ and
    $w=0$ on $\p\Sigma$. Let $\psi=\log w$, and assume
    \begin{equation}
        -\Delta\psi\leq\theta|\nabla\psi|^2\ \text{ on }\Sigma
        \label{eq:IBPw hypothesis}
    \end{equation}
    for some $\theta\in(0,1)$. Then for
    $v\in W^{1,2}_{loc}(\Sigma,\mr^N)\cap L^2(\mu_w)$ with
    $\mathcal A:=\int_\Sigma\|\nabla v\|^2\,d\mu_w<\infty$, the quantity
    $H:=\int_\Sigma|\nabla\psi|^2|v|^2\,d\mu_w$ is finite and
    \begin{equation}
        H =\int_\Sigma(-\Delta\psi)|v|^2\,d\mu_w -\int_\Sigma\langle\nabla\psi,\nabla|v|^2\rangle\,d\mu_w +\alpha\int_\Sigma|v|^2\,d\mu_w.
        \label{eq:IBPw identity}
    \end{equation}
    In particular, $\nabla\psi\in L^2(\mu_w)$ and $\int_\Sigma|\nabla\psi|^2d\mu_w\leq\frac{\alpha}{1-\theta}$.
\end{lemma}

\begin{proof}
   We prove the identity by using a smooth cutoff function near the boundary, and applying the divergence theorem. Fix a smooth nonincreasing $\eta:\mr\to[0,1]$ with $\eta=1$ on
    $(-\infty,1]$ and $\eta=0$ on $[2,\infty)$, and a smooth
    $\zeta:[0,\infty)\to[0,1]$ with $\zeta=1$ on $[0,1]$, $\zeta$ nonincreasing, $\zeta=0$ on
    $[2,\infty)$ and $|\zeta'|\leq2$. For $m,R\geq1$ set
    \begin{equation*}
        \chi_m=\eta\!\left(\frac{-\log w}{m}\right),
        \qquad
        \zeta_R(x)=\zeta\!\left(\frac{|x|}{R}\right).
    \end{equation*}
    Then $\chi_m\zeta_R^2$ has compact support
    in $\Sigma$. Because $\eta$ is nonincreasing, the chain rule gives
    \begin{equation}
        \nabla\chi_m=-\frac1m\eta'\!\left(\frac{-\log w}{m}\right)\nabla\psi,
        \qquad
        \langle\nabla\chi_m,\nabla\psi\rangle
        =-\frac1m\eta'\!\left(\frac{-\log w}{m}\right)|\nabla\psi|^2\geq0.
        \label{eq:IBPw sign}
    \end{equation}

    Consider the vector field
    \begin{equation*}
        F_{m,R}=\chi_m\zeta_R^2\,|v|^2\,w\,e^{-\frac{|x|^2}{2}}\nabla\psi.
    \end{equation*}
    Since $F_{m,R}\in W^{1,1}(\Sigma)$ has compact support in $\Sigma$, the divergence theorem applies. Thus
    \begin{equation}
        \int_\Sigma\text{div}F_{m,R}\,dx=0.
        \label{eq:IBPw divergence}
    \end{equation}
    We expand the divergence.
    \begin{align*}
        \operatorname{div}F_{m,R}
        &=we^{-\frac{|x|^2}{2}}\Big(\zeta_R^2|v|^2\langle\nabla\chi_m,\nabla\psi\rangle+2\chi_m\zeta_R|v|^2\langle\nabla\zeta_R,\nabla\psi\rangle+\chi_m\zeta_R^2\langle\nabla|v|^2,\nabla\psi\rangle\Big)\\
        &\quad+we^{-\frac{|x|^2}{2}}\chi_m\zeta_R^2|v|^2\Big(\langle\nabla\psi-x,\nabla\psi\rangle+\Delta\psi\Big)\\
        &=we^{-\frac{|x|^2}{2}}\Big(\zeta_R^2|v|^2\langle\nabla\chi_m,\nabla\psi\rangle+2\chi_m\zeta_R|v|^2\langle\nabla\zeta_R,\nabla\psi\rangle+\chi_m\zeta_R^2\langle\nabla|v|^2,\nabla\psi\rangle\Big)\\
        &\quad+we^{-\frac{|x|^2}{2}}\chi_m\zeta_R^2|v|^2\Big(|\nabla\psi|^2-\alpha+\Delta\psi\Big).
    \end{align*}
    Here we used $\nabla\left(we^{-\frac{|x|^2}{2}}\right)=we^{-\frac{|x|^2}{2}}(\nabla\psi-x)$ and, in the last equality, Euler's identity $x\cdot\nabla\psi=\alpha$ for homogeneous functions. Integrating against $Z_w^{-1}dx$, we obtain
    \begin{equation}
        H_{m,R}=T_{m,R}-M_{m,R}+\alpha V_{m,R}-E_{m,R}-Q_{m,R},
        \label{eq:IBPw with cutoffs}
    \end{equation}
    where
    \begin{align*}
        H_{m,R} &= \int_\Sigma \chi_m\zeta_R^2|\nabla\psi|^2|v|^2d\mu_w \\
        T_{m,R} &= \int_\Sigma \chi_m\zeta_R^2(-\Delta\psi)|v|^2d\mu_w \\
        V_{m,R} &= \int_\Sigma \chi_m\zeta_R^2|v|^2d\mu_w\\
        M_{m,R} &= \int_\Sigma \chi_m\zeta_R^2\langle\nabla\psi,\nabla|v|^2\rangle
        d\mu_w \\
        E_{m,R} &= \int_\Sigma \zeta_R^2|v|^2\langle\nabla\chi_m,\nabla\psi\rangle
        d\mu_w\\
        Q_{m,R} &= 2\int_\Sigma \chi_m\zeta_R|v|^2\langle\nabla\zeta_R,\nabla\psi\rangle
        d\mu_w.
    \end{align*}

    On the compact support of $\chi_m\zeta_R^2$ the continuous function
    $|\nabla\psi|$ is bounded, so $$H_{m,R}\leq(\sup|\nabla\psi|)^2 \int_\Sigma |v|^2 d\mu_w<\infty.$$ Every other term in \eqref{eq:IBPw with cutoffs} is finite as well, and we now estimate each term.
    \begin{equation}
        T_{m,R}\leq\theta\int_\Sigma \chi_m\zeta_R^2|\nabla\psi|^2|v|^2d\mu_w
        =\theta H_{m,R}.
        \label{eq:IBPw T}
    \end{equation}
    Since $\nabla|v|^2=2\sum_kv_k\nabla v_k$, the Cauchy--Schwarz inequality
    gives
    \begin{align*}\left|\sum_kv_k \langle\nabla v_k, \nabla \psi \rangle\right| &\leq \left(\sum_k |v_k|^2\right)^{\frac{1}{2}} \left(\sum_k |\langle\nabla v_k, \nabla \psi \rangle|^2\right)^{\frac{1}{2}}\\ &\leq |v| \left(| \nabla\psi|^2 \underbrace{\sum_k |\nabla v_k|^2 }_{=\|\nabla v\|^2}\right)^{\frac{1}{2}} =|v| | \nabla\psi|  \|\nabla v\|. \end{align*}
   Then we have
    \begin{equation*}
        |M_{m,R}|
        =\left|\int_\Sigma\chi_m\zeta_R^2\langle\nabla\psi,\nabla|v|^2\rangle
        d\mu_w\right|
        \leq2\int_\Sigma\chi_m\zeta_R^2\,|\nabla\psi||v|\,\|\nabla v\|
        d\mu_w.
    \end{equation*}
    We write the integrand as
    \begin{equation*}
        \chi_m\zeta_R^2\,|\nabla\psi||v|\,\|\nabla v\|
        =\left(\chi_m\zeta_R\,|\nabla\psi||v|\right)
        \left(\zeta_R\,\|\nabla v\|\right)
    \end{equation*}
    and apply the Cauchy--Schwarz inequality again
    \begin{equation*}
        |M_{m,R}|
        \leq2\left(\int_\Sigma\chi_m^2\zeta_R^2|\nabla\psi|^2|v|^2d\mu_w\right)
        ^{1/2}
        \left(\int_\Sigma\zeta_R^2\|\nabla v\|^2d\mu_w\right)^{1/2}.
    \end{equation*}
    Since $0\leq\chi_m\leq1$ we have $\chi_m^2\leq\chi_m$ and since $\zeta_R^2\leq1$, we have
    \begin{equation}
        |M_{m,R}|
        \leq2\left(\int_\Sigma\chi_m\zeta_R^2|\nabla\psi|^2|v|^2d\mu_w\right)
        ^{1/2}
        \left(\int_\Sigma\zeta_R^2\|\nabla v\|^2d\mu_w\right)^{1/2}
        \leq2H_{m,R}^{1/2}\mathcal A^{1/2}.
        \label{eq:IBPw M}
    \end{equation}
    Since $|\nabla\zeta_R|\leq2/R$ we have $|\langle\nabla\zeta_R,\nabla\psi\rangle|\leq|\nabla\zeta_R||\nabla\psi|
    \leq\frac{2}{R}|\nabla\psi|$. Then
    \begin{equation*}
        |Q_{m,R}|
        =\left|2\int_\Sigma\chi_m\zeta_R|v|^2\langle\nabla\zeta_R,\nabla\psi
        \rangle d\mu_w\right|
        \leq\frac{4}{R}\int_\Sigma\chi_m\zeta_R\,|v|^2\,|\nabla\psi|\,d\mu_w.
    \end{equation*}
    We apply the Cauchy--Schwarz inequality to $\left(\chi_m^{1/2}\zeta_R\,|\nabla\psi||v|\right)
        \left(\chi_m^{1/2}|v|\right)$ to get
    \begin{align}
        |Q_{m,R}| &\leq\frac{4}{R}
        \left(\int_\Sigma\chi_m\zeta_R^2|\nabla\psi|^2|v|^2d\mu_w\right)^{1/2}
        \left(\int_\Sigma\chi_m|v|^2d\mu_w\right)^{1/2} \notag\\
        &\leq \frac{4}{R}
        \left(\int_\Sigma \chi_m\zeta_R^2|\nabla\psi|^2|v|^2d\mu_w\right)^{1/2}
        \left(\int_\Sigma|v|^2d\mu_w\right)^{1/2} \notag\\
        &= \frac{4}{R}H_{m,R}^{1/2} J^{1/2}, \label{eq:IBPw Q}
    \end{align}
    where $J = \int_\Sigma |v|^2 d\mu_w$.
    \begin{equation}
        V_{m,R} = \int_\Sigma \chi_m\zeta_R^2|v|^2d\mu_w \leq \int_\Sigma |v|^2d\mu_w = J. \label{eq:IBPw V}
    \end{equation}
    And
    \begin{equation}
        E_{m,R} = \int_\Sigma \zeta_R^2|v|^2\langle\nabla\chi_m,\nabla\psi\rangle
        d\mu_w \geq 0 \qquad \text{by \eqref{eq:IBPw sign}} \label{eq:IBPw E}
    \end{equation}
    Using these bounds in \eqref{eq:IBPw with cutoffs} we get
    \begin{equation}
        H_{m,R} \leq \theta H_{m,R} + 2 H_{m,R}^{1/2} \mathcal A^{1/2} + \alpha J + \frac{4}{R}H_{m,R}^{1/2} J^{1/2}. \notag
    \end{equation}
    Rewriting we get
    \begin{equation}
        (1-\theta)H_{m,R}
        \leq\left(2\mathcal A^{1/2}+\frac{4}{R}J^{1/2}\right)H_{m,R}^{1/2}+\alpha J.
        \notag
    \end{equation}
    Let $y = H_{m,R}^{1/2}$. Then the inequality becomes
    \begin{equation}
        (1-\theta)y^2 - \left(2\mathcal A^{1/2}+\frac{4}{R}J^{1/2}\right)y - \alpha J
        \leq0. \label{eq:IBPw quadratic}
    \end{equation}
    Because $y$ is finite and $\theta < 1$, \eqref{eq:IBPw quadratic} is a quadratic
    inequality in $y$ with positive leading coefficient. Solving we get
    \begin{align}
        y
        & \leq\frac{\left(2\mathcal A^{1/2}+\frac{4}{R}J^{1/2}\right)
        +\sqrt{\left(2\mathcal A^{1/2}+\frac{4}{R}J^{1/2}\right)^2+4(1-\theta)\alpha J}}
        {2(1-\theta)} \label{eq:IBPw root} \\
        & \leq\frac{\left(2\mathcal A^{1/2}+4J^{1/2}\right)
        +\sqrt{\left(2\mathcal A^{1/2}+4J^{1/2}\right)^2+4(1-\theta)\alpha J}}
        {2(1-\theta)}.
        \notag
    \end{align}
    This bound on $y$ is uniform for all $m,R \geq 1$. Thus
    \begin{equation}
        H_{m,R}\leq K_0, \label{eq:IBPw uniform}
    \end{equation}
    where $K_0$ is the square of the right-hand side of \eqref{eq:IBPw root} with $R=1$.

    Now let $R\to\infty$. By \eqref{eq:IBPw Q} and \eqref{eq:IBPw uniform},
    $|Q_{m,R}|\leq\frac{4}{R}K_0^{1/2}J^{1/2}\to0$. As $\zeta_R\uparrow1$,
    monotone convergence gives $H_{m,R}\to H_m$, $V_{m,R}\to V_m$ and
    $E_{m,R}\to E_m$, while dominated convergence gives $T_{m,R}\to T_m$ and
    $M_{m,R}\to M_m$, the same integrals without the factor $\zeta_R$. Here the
    dominating functions $\theta\chi_m|\nabla\psi|^2|v|^2$ (this is a dominating function because $0\leq-\Delta\psi\leq\theta|\nabla\psi|^2$, where $-\Delta\psi=\operatorname{tr}(-\nabla^2\psi)\geq0$ by log-concavity) and
    $2\chi_m|\nabla\psi||v|\|\nabla v\|$ are integrable by
    \eqref{eq:IBPw uniform} and $\mathcal A<\infty$. Thus
    \begin{equation}
        H_m=T_m-M_m+\alpha V_m-E_m,
        \qquad H_m\leq K_0. \label{eq:IBPw identity in m}
    \end{equation}

    Finally let $m\to\infty$. Since $\chi_m\uparrow1$, monotone convergence
    applied to $H_m$ gives $H=\lim_m H_m\leq K_0<\infty$, which gives us the
    finiteness of $H$. Knowing $H<\infty$, we have $V_m\to J$ by monotone convergence,
    while $T_m$ and $M_m$ converge to the first two integrals on the right of
    \eqref{eq:IBPw identity} by dominated convergence, now dominated using
    $H<\infty$. Since $\nabla\chi_m$ is supported in
    $\{m\leq-\log w\leq2m\}$,
    \begin{equation*}
        0\leq E_m
        \leq\frac{\|\eta'\|_\infty}{m}
        \int_{\{-\log w\geq m\}}|\nabla\psi|^2|v|^2d\mu_w
        \leq\frac{\|\eta'\|_\infty}{m}H
        \longrightarrow0.
    \end{equation*}
    Taking the limit in \eqref{eq:IBPw identity in m} gives \eqref{eq:IBPw identity}.

    For the last claim, let $v$ be a constant unit vector. Then $\mathcal A=0$ and $\int_\Sigma|v|^2d\mu_w=1$, so \eqref{eq:IBPw identity} and $-\Delta\psi\leq\theta|\nabla\psi|^2$ give $H\leq\theta H+\alpha$. Since $H<\infty$, we get $\int_\Sigma|\nabla\psi|^2d\mu_w=H\leq\frac{\alpha}{1-\theta}$.
\end{proof}

\begin{proof}[Proof of Theorem~\ref{thm:BL stability homogeneous}]
    First, the hypothesis $\Delta(w^\theta)\geq0$ is the same as $-\Delta\psi\leq\theta|\nabla\psi|^2$ on $\Sigma$, since
    \begin{equation*}
        \Delta\left(w^\theta\right)
        =\Delta\left(e^{\theta\psi}\right)
        =\theta e^{\theta\psi}\left(\Delta\psi+\theta|\nabla\psi|^2\right)
    \end{equation*}
    and $\theta e^{\theta\psi}>0$ on $\Sigma$. So Lemma~\ref{lemma:IBPw identity} applies.

    Since $\mu_w$ satisfies $P(\lambda_w)$ and $\int_\Sigma f\,d\mu_w=0$, there is a unique solution $u$ of the Poisson equation $-L_wu = f$ with $\int_\Sigma u\,d\mu_w=0$. As in the monomial case (the boundary terms vanish since $w=0$ on $\p\Sigma$) we have
    \begin{align}
        \int_\Sigma \nabla u \, d\mu_w = \int_\Sigma fx \, d\mu_w = a.
    \end{align}
    Next we define $h = u - a\cdot x$ and it follows that $$v = \nabla h = \nabla u - a,$$
    and also note that
    \begin{equation}
        \int_\Sigma v \, d\mu_w = \int_\Sigma \nabla u \, d\mu_w - a = 0.
    \end{equation}
    Exactly as in the monomial case, we have the dual form of the deficit
    \begin{equation}
        \delta_{BL}(f) = \int_\Sigma \|\nabla^2 u\|^2 \, d\mu_w + \int_\Sigma \langle B^{-1}(-L_w)\nabla u, (-L_w)\nabla u\rangle \, d\mu_w. \label{BLw:deficit}
    \end{equation}
    Since $(-L_w)(a) = 0$ and $\nabla v = \nabla^2 u$ we can rewrite \eqref{BLw:deficit} in terms of $v$:
    \begin{equation}
        \delta_{BL}(f) = \underbrace{\int_\Sigma \|\nabla v\|^2 \, d\mu_w}_{\mathcal A} + \underbrace{\int_\Sigma \langle B^{-1}(-L_w)v, (-L_w)v\rangle \, d\mu_w}_{\mathcal C}. \label{BLw:deficit in v}
    \end{equation}
    Also note that
    \begin{align}
        -L_wh &= -L_w(u - a\cdot x) \notag \\
        &= f + L_w(a\cdot x) \notag \\
        &= f + (\Delta - \nabla W\cdot\nabla)(a\cdot x) \notag \\
        &= f - \langle\nabla W, a\rangle. \label{BLw:Lh}
    \end{align}
    So we get
    \begin{align}
        \int_\Sigma |f - \langle\nabla W, a\rangle|^2 \, d\mu_w
        &= \int_\Sigma |L_wh|^2 \, d\mu_w \notag \\
        &= \int_\Sigma \Gamma_2(h) \, d\mu_w \notag \\
        &= \int_\Sigma \|\nabla^2 h\|^2 + \langle \nabla^2 W\nabla h, \nabla h\rangle \, d\mu_w \notag \\
        &= \underbrace{\int_\Sigma \|\nabla v\|^2 \, d\mu_w}_{\mathcal A} + \underbrace{\int_\Sigma \langle Bv, v\rangle \, d\mu_w}_{\mathcal D}, \label{BLw:distance to opt}
    \end{align}
    where $B = \nabla^2 W = \mathbb{I}-\nabla^2\psi$.

    By the definition of $\lambda_w$, the measure $\mu_w$ satisfies the Poincar\'e inequality $P(\lambda_w)$. Applying it to each component $v_i$ and using $\int_\Sigma v\,d\mu_w = 0$ we have
    \begin{equation}
        \int_\Sigma |v|^2 \, d\mu_w \leq \frac{1}{\lambda_w}\int_\Sigma \|\nabla v\|^2 \, d\mu_w = \frac{\mathcal A}{\lambda_w}. \label{BLw:poincare}
    \end{equation}
    By Lemma \ref{lemma:IBPw identity} the quantity $H = \int_\Sigma |\nabla\psi|^2|v|^2 \, d\mu_w$ is finite and
    \begin{equation}
        H = \int_\Sigma (-\Delta\psi)|v|^2 \, d\mu_w - \int_\Sigma \langle\nabla\psi, \nabla|v|^2\rangle \, d\mu_w + \alpha\int_\Sigma |v|^2 \, d\mu_w. \label{BLw:hardy}
    \end{equation}
    Bounding the first term of \eqref{BLw:hardy} by the hypothesis $-\Delta\psi \leq \theta|\nabla\psi|^2$, the second by Cauchy--Schwarz, and the third by \eqref{BLw:poincare}, we get
    \begin{equation}
        H \leq \theta H + 2H^{\frac{1}{2}}\mathcal A^{\frac{1}{2}} + \frac{\alpha}{\lambda_w}\mathcal A. \notag
    \end{equation}
    If $\mathcal A=0$ then $v=0$ and there is nothing to prove. Otherwise let $y = \left(\frac{H}{\mathcal A}\right)^{\frac{1}{2}}$. Then this can be rewritten as
    \begin{equation}
        (1-\theta)y^2 - 2y - \frac{\alpha}{\lambda_w} \leq 0. \label{BLw:quadratic}
    \end{equation}
    Solving \eqref{BLw:quadratic} we get
    \begin{equation}
        y \leq \frac{1 + \sqrt{1 + \frac{(1-\theta)\alpha}{\lambda_w}}}{1-\theta} =: Y_w. \label{BLw:ybound}
    \end{equation}
    Thus
    \begin{equation}
        \frac{H}{\mathcal A} = y^2 \leq Y_w^2, \qquad H \leq Y_w^2 \mathcal A. \label{BLw:Hbound}
    \end{equation}
    Since $-\nabla^2\psi \succeq 0$, we have $\langle -\nabla^2\psi\,v, v\rangle \leq \operatorname{tr}(-\nabla^2\psi)|v|^2 = (-\Delta\psi)|v|^2 \leq \theta|\nabla\psi|^2|v|^2$. Combining \eqref{BLw:poincare} and \eqref{BLw:Hbound} we get the estimate
    \begin{align}
        \underbrace{\int_\Sigma \langle Bv, v\rangle \, d\mu_w}_{\mathcal D}
        &= \int_\Sigma |v|^2 \, d\mu_w + \int_\Sigma \langle -\nabla^2\psi\,v, v\rangle \, d\mu_w \notag \\
        &\leq \frac{\mathcal A}{\lambda_w} + \theta H \notag \\
        &\leq \left(\frac{1}{\lambda_w} + \theta Y_w^2\right)\underbrace{\int_\Sigma \|\nabla v\|^2 \, d\mu_w}_{\mathcal A}. \label{BLw:ADbound}
    \end{align}
    Here $\kappa_w^* = \frac{1}{\lambda_w} + \theta Y_w^2=\frac{1}{\kappa_w}$. Next consider the term
    \begin{align}
        \mathcal A &= \int_\Sigma \|\nabla v\|^2 \, d\mu_w = \int_\Sigma \langle -L_wv, v\rangle \, d\mu_w \notag \\
        &= \int_\Sigma \langle B^{\frac{1}{2}}B^{-\frac{1}{2}}(-L_wv), v\rangle \, d\mu_w \notag \\
        &= \int_\Sigma \langle B^{-\frac{1}{2}}(-L_wv), B^{\frac{1}{2}}v\rangle \, d\mu_w \notag \\
        \text{by C-S}\quad &\leq \left(\int_\Sigma \langle B^{-\frac{1}{2}}(-L_wv), B^{-\frac{1}{2}}(-L_wv)\rangle \, d\mu_w\right)^{\frac{1}{2}}\left(\int_\Sigma \langle B^{\frac{1}{2}}v, B^{\frac{1}{2}}v\rangle \, d\mu_w\right)^{\frac{1}{2}} \notag \\
        &= \left(\int_\Sigma \langle B^{-1}(-L_wv), (-L_wv)\rangle \, d\mu_w\right)^{\frac{1}{2}}\left(\int_\Sigma \langle Bv, v\rangle \, d\mu_w\right)^{\frac{1}{2}} \notag \\
        &= \mathcal C^{\frac{1}{2}}\mathcal D^{\frac{1}{2}}. \label{BLw:helper8}
    \end{align}
    In other words, \eqref{BLw:helper8} gives us the bound
    \begin{equation}
        \mathcal A^2 \leq \mathcal C\mathcal D \iff \mathcal C \geq \frac{\mathcal A^2}{\mathcal D} \implies \mathcal C \geq \frac{1}{\kappa_w^*}\mathcal A \quad \text{by \eqref{BLw:ADbound}} \label{BLw:ACbound}
    \end{equation}
    Putting everything together we have
    \begin{align}
        \delta_{BL}(f) &= \mathcal A + \mathcal C \qquad \text{by \eqref{BLw:deficit in v}} \notag \\
        &\geq \mathcal A + \frac{1}{\kappa_w^*}\mathcal A \qquad \text{by \eqref{BLw:ACbound}} \notag \\
        &\geq \frac{1}{\kappa_w^*}\mathcal D + \frac{1}{\kappa_w^*}\mathcal A \qquad \text{by \eqref{BLw:ADbound}} \notag \\
        &= \frac{1}{\kappa_w^*}\int_\Sigma |f - \langle\nabla W, a\rangle|^2 \, d\mu_w \qquad \text{by \eqref{BLw:distance to opt}} \notag \\
        &\geq \frac{1}{\kappa_w^*}\inf_{c\in\mr^N} \int_\Sigma |f - \langle\nabla W, c\rangle|^2 \, d\mu_w. \label{BLw:stability}
    \end{align}
    This gives the $L^2$ stability.
\end{proof}

We end this section by checking the hypotheses of Theorem~\ref{thm:BL stability homogeneous} for the monomial weight.
\begin{lemma}\label{lem:monomial hypotheses}
    Let $w=x^A$ with $\alpha_i>0$ for all $i$. Then $w$ is log-concave, and the following are equivalent: (i) $-\Delta\log w\leq\theta|\nabla\log w|^2$ for some $\theta\in(0,1)$, (ii) $\nabla\log w\in L^2(\mu_A)$, (iii) $\alpha_i>1$ for all $i$. In this case the best constant in (i) is $\theta=\max_i\frac1{\alpha_i}$.
\end{lemma}
\begin{proof}
    Write $\psi=\log w=\sum_i\alpha_i\log x_i$. Then
    \begin{equation}
        \p_i\psi=\frac{\alpha_i}{x_i},
        \qquad
        \p_i^2\psi=-\frac{\alpha_i}{x_i^2},
        \qquad
        \text{so}\qquad
        |\nabla\psi|^2=\sum_i\frac{\alpha_i^2}{x_i^2},
        \qquad
        -\Delta\psi=\sum_i\frac{\alpha_i}{x_i^2}.
        \label{eq:monomial derivatives}
    \end{equation}
    In particular $-\nabla^2\psi=\operatorname{diag}(\alpha_i/x_i^2)
    \succeq0$, so $w$ is log-concave for all $\alpha_i>0$.

    Comparing the two sums in
    \eqref{eq:monomial derivatives} termwise, $-\Delta\psi\leq\theta|\nabla\psi|^2$
    holds if $\alpha_i\leq\theta\alpha_i^2$ for every $i$, that is,
    $\theta\geq1/\alpha_i$. Conversely, letting $x_j\to0$ for one index $j$ while the other coordinates are fixed shows that $\theta\geq1/\alpha_j$ is necessary. The smallest constant is therefore
    \begin{equation}
        \theta=\max_i\frac{1}{\alpha_i}
        \label{eq:monomial theta}
    \end{equation}
    and the requirement $\theta<1$ is precisely $\alpha_i>1$ for all $i$.

    For the integrability, since $\mu_A$ is a product measure,
    \begin{equation}
        \int_{\mr_*^N}|\nabla\psi|^2d\mu_A
        =\sum_i\alpha_i^2\int_{\mr_*^N}\frac{1}{x_i^2}d\mu_A,
        \qquad
        \int_0^\infty\frac{1}{x_i^2}\,x_i^{\alpha_i}e^{-x_i^2/2}dx_i
        =\int_0^\infty x_i^{\alpha_i-2}e^{-x_i^2/2}dx_i.
        \label{eq:monomial FI}
    \end{equation}
    The one-dimensional integral is finite if and only if
    $\alpha_i-2>-1$, that is $\alpha_i>1$. Hence $\nabla\psi\in L^2(\mu_A)$ is again equivalent to
    $\alpha_i>1$ for all $i$.
\end{proof}
\section{ICD for homogeneous measures}\label{sec:ICD for hom}
In this section $\Sigma\subset\mr^N$ is an open convex cone, $w\in C(\overline\Sigma)\cap C^2(\Sigma)$ is positive in $\Sigma$ and homogeneous of degree $\alpha\geq0$, and $\mu_w$, $L_w$ are as in the introduction. We define the effective dimension $D = N + \alpha$.
\begin{theorem}[Improved ICD] \label{thm: improved ICD}
    Assume that for some $k_w>0$
    \begin{equation}
        |x|^2(-\nabla^2 \log w) \geq k_w \mathbb{I} \quad\text{on }\Sigma. \label{codition: scale invariant hessian lower bound}
    \end{equation}
    Then $\mu_w$ satisfies the integrated curvature-dimension condition $ICD(1+s_w, \infty)$ for all $u\in C^2(\overline\Sigma)$ whose derivatives of order at most $2$ grow at most polynomially, where
    \begin{equation}
        s_w=\frac{\sqrt{(D-2)^2+4k_w}-(D-2)}{2} > 0.
        \label{eq:sw definition}
    \end{equation}
\end{theorem}
As before, we first prove a lemma.
\begin{lemma}[Radial lemma]\label{lemma:radial lemma}
    Let $D>0$ and $k_w>0$, and let $s_w$ be given by \eqref{eq:sw definition}, so that
    \begin{equation}
        s_w(s_w+D-2)=k_w.
        \label{eq:sw relation}
    \end{equation}
    Let $g\in C^1(0,\infty)$ satisfy $|g(r)|^2r^{D-2}e^{-\frac{r^2}{2}}\to0$ as $r\to0$ and as $r\to\infty$. Then
    \begin{equation}
        \int_0^\infty
        \left(|g'(r)|^2+\frac{k_w}{r^2}|g(r)|^2\right)
        r^{D-1}e^{-\frac{r^2}{2}}dr
        \geq
        s_w\int_0^\infty
        |g(r)|^2r^{D-1}e^{-\frac{r^2}{2}}dr.
        \label{ineq:radial inequality}
    \end{equation}
    More precisely, we have the identity
    \begin{align}
        &\int_0^\infty
        \left(|g'(r)|^2+\frac{k_w}{r^2}|g(r)|^2\right)
        r^{D-1}e^{-\frac{r^2}{2}}dr
        -
        s_w\int_0^\infty
        |g(r)|^2r^{D-1}e^{-\frac{r^2}{2}}dr \notag\\
        &\qquad =
        \int_0^\infty
        r^{2s_w}
        \left|
        \left(\frac{g(r)}{r^{s_w}}\right)'
        \right|^2
        r^{D-1}e^{-\frac{r^2}{2}}dr.
        \label{eq:radial identity}
    \end{align}
    Equality in \eqref{ineq:radial inequality} holds if and only if
    \begin{equation*}
        g(r)=cr^{s_w}
    \end{equation*}
    for some $c\in\mathbb R$.
\end{lemma}

\begin{proof}
    Let
    \begin{equation*}
        d\nu(r)=r^{D-1}e^{-\frac{r^2}{2}}dr.
    \end{equation*}
    Write $g(r)=r^{s_w}h(r)$. We have
    \begin{align*}
        g'(r)
        &=s_wr^{s_w-1}h(r)+r^{s_w}h'(r)\\
        h'(r)
        &=\frac{g'(r)-s_wr^{s_w-1}h(r)}{r^{s_w}}\\
        &=r^{-s_w}\left(g'(r)-\frac{s_w}{r}g(r)\right).
    \end{align*}
    Therefore,
    \begin{align*}
        r^{2s_w}|h'(r)|^2
        &=
        \left|g'(r)-\frac{s_w}{r}g(r)\right|^2\\
        &=
        |g'(r)|^2
        -
        \frac{2s_w}{r}g(r)g'(r)
        +
        \frac{s_w^2}{r^2}|g(r)|^2.
    \end{align*}
    Integrating with respect to $\nu$, we obtain
    \begin{align}
        \int_0^\infty r^{2s_w}|h'(r)|^2d\nu(r)
        &=
        \int_0^\infty |g'(r)|^2d\nu(r)
        -
        2s_w\int_0^\infty\frac{g(r)g'(r)}{r}d\nu(r) \notag\\
        &\quad+
        s_w^2\int_0^\infty\frac{|g(r)|^2}{r^2}d\nu(r).
        \label{eq:radial expansion}
    \end{align}
    We compute the middle term using
    \begin{equation*}
        2g(r)g'(r)=\left(|g(r)|^2\right)'.
    \end{equation*}
    Since
    \begin{equation*}
        \frac{1}{r}d\nu(r)
        =
        r^{D-2}e^{-\frac{r^2}{2}}dr,
    \end{equation*}
    integration by parts gives
    \begin{align*}
        -2s_w\int_0^\infty\frac{g(r)g'(r)}{r}d\nu(r)
        &=
        -s_w\int_0^\infty
        \left(|g(r)|^2\right)'
        r^{D-2}e^{-\frac{r^2}{2}}dr\\
        &=
        -s_w
        \left[
        |g(r)|^2r^{D-2}e^{-\frac{r^2}{2}}
        \right]_0^\infty\\
        &\quad+
        s_w\int_0^\infty
        |g(r)|^2
        \left(
        r^{D-2}e^{-\frac{r^2}{2}}
        \right)'dr.
    \end{align*}
    The boundary term vanishes by the assumption on $g$. (If $\int_0^\infty |g'|^2d\nu$ or $\int_0^\infty \frac{|g|^2}{r^2}d\nu$ is infinite, there is nothing to prove. Otherwise all integrals here are finite, as in Lemma~\ref{lemma1}.) Moreover,
    \begin{align*}
        \left(
        r^{D-2}e^{-\frac{r^2}{2}}
        \right)'
        &=
        (D-2)r^{D-3}e^{-\frac{r^2}{2}}
        -
        r^{D-1}e^{-\frac{r^2}{2}}\\
        &=
        \left(\frac{D-2}{r^2}-1\right)
        r^{D-1}e^{-\frac{r^2}{2}}.
    \end{align*}
    It follows that
    \begin{align}
        -2s_w\int_0^\infty\frac{g(r)g'(r)}{r}d\nu(r)
        &=
        s_w(D-2)
        \int_0^\infty\frac{|g(r)|^2}{r^2}d\nu(r)\notag\\
        &\quad-
        s_w\int_0^\infty |g(r)|^2d\nu(r).
        \label{eq:radial cross term}
    \end{align}
    Substituting \eqref{eq:radial cross term} into \eqref{eq:radial expansion}, we obtain
    \begin{align*}
        \int_0^\infty r^{2s_w}|h'(r)|^2d\nu(r)
        &=
        \int_0^\infty |g'(r)|^2d\nu(r)\\
        &\quad+
        \left(s_w^2+s_w(D-2)\right)
        \int_0^\infty\frac{|g(r)|^2}{r^2}d\nu(r)\\
        &\quad-
        s_w\int_0^\infty |g(r)|^2d\nu(r).
    \end{align*}
    Using \eqref{eq:sw relation}, we have
    \begin{equation*}
        s_w^2+s_w(D-2)=s_w(s_w+D-2)=k_w.
    \end{equation*}
    Therefore,
    \begin{align*}
        \int_0^\infty r^{2s_w}|h'(r)|^2d\nu(r)
        &=
        \int_0^\infty
        \left(
        |g'(r)|^2+\frac{k_w}{r^2}|g(r)|^2
        \right)d\nu(r)\\
        &\quad-
        s_w\int_0^\infty |g(r)|^2d\nu(r).
    \end{align*}
    Since
    \begin{equation*}
        h(r)=\frac{g(r)}{r^{s_w}},
    \end{equation*}
    this proves \eqref{eq:radial identity}. Since the right hand side of \eqref{eq:radial identity} is nonnegative, we obtain \eqref{ineq:radial inequality}.

    Equality holds if and only if
    \begin{equation*}
        \left(\frac{g(r)}{r^{s_w}}\right)'=0,
    \end{equation*}
    which is equivalent to
    \begin{equation*}
        g(r)=cr^{s_w}.
    \end{equation*}
\end{proof}
\begin{proof}[Proof of Theorem \ref{thm: improved ICD}]
    Let $v = \nabla u$. Then
    \begin{align}
        \|\nabla^2 u\|^2 &= \sum_{ij} |\p_{ij}u|^2 \notag\\
        &= \sum_{ij} |\p_jv_i|^2 \notag \\
        &= \sum_i |\nabla v_i|^2. \label{ICDhelper1}
    \end{align}
    Now let $x = r\sigma$ where $\sigma \in S^{N-1}$ and $r > 0$. Then
    \begin{align}
        |\p_r v_i(r\sigma)|^2 &= \left|\sum_j \p_j v_i(r\sigma) \cdot\sigma_j \right|^2 \notag \\
        &\leq \left(\sum_j \sigma_j^2 \right) \left(\sum_j |\p_jv_i|^2\right) \notag \\
        &= |\nabla v_i|^2.  \label{ICDhelper2}
    \end{align}
    Summing \eqref{ICDhelper2} over $i$ and using \eqref{ICDhelper1}, we get
    \begin{align}
        \| \nabla^2 u\|^2 &\geq \sum_i |\p_r v_i(r\sigma)|^2 \notag \\
        &= |\p_r  v|^2. \label{ICDhelper3}%
    \end{align}
    Now the assumption \eqref{codition: scale invariant hessian lower bound} together with \eqref{ICDhelper3} gives us
    \begin{equation}
        \|\nabla^2 u\|^2 -\nabla^2\log w(\nabla u, \nabla u) \geq |\p_r  v|^2 + \frac{k_w}{|x|^2}|v|^2. \label{ICDhelper4}%
    \end{equation}
     Now we pass to polar coordinates. Since $w$ is homogeneous of degree
    $\alpha$ we have $w(r\sigma)=r^\alpha w(\sigma)$, and since
    $dx=r^{N-1}drd\sigma$ it follows that
    \begin{align}
        d\mu_w
        &=\frac{1}{Z_w}w(r\sigma)e^{-\frac{r^2}{2}}r^{N-1}drd\sigma \notag\\
        &=\frac{1}{Z_w}w(\sigma)r^{N+\alpha-1}e^{-\frac{r^2}{2}}drd\sigma \notag\\
        &=\frac{1}{Z_w}w(\sigma)r^{D-1}e^{-\frac{r^2}{2}}drd\sigma,
        \label{eq:polar measure}
    \end{align}
    so that the radial factor is exactly the weight $r^{D-1}e^{-r^2/2}$
    appearing in Lemma \ref{lemma:radial lemma}.

    Integrating \eqref{ICDhelper4} against $\mu_w$ and using
    \eqref{eq:polar measure} we obtain
    \begin{align}
        &\int_\Sigma\|\nabla^2u\|^2d\mu_w
        -\int_\Sigma\nabla^2\log w(\nabla u,\nabla u)d\mu_w \notag\\
        &\qquad\geq
        \frac{1}{Z_w}\int_{\Sigma\cap S^{N-1}}w(\sigma)
        \int_0^\infty\left(|\p_rv(r\sigma)|^2
        +\frac{k_w}{r^2}|v(r\sigma)|^2\right)
        r^{D-1}e^{-\frac{r^2}{2}}drd\sigma.
        \label{ICDhelper5}
    \end{align}
    Fix $\sigma\in\Sigma\cap S^{N-1}$ and $i\in\{1,\dots,N\}$, and apply
    Lemma \ref{lemma:radial lemma} to $g(r)=v_i(r\sigma)$. The assumption of the lemma holds at infinity because the derivatives of $u$ grow at most polynomially. At $r=0$ it holds if $v_i(0)=0$ or $D>2$, and if $v_i(0)\neq0$ and $D\leq2$, then the left-hand side below is $+\infty$. Since
    $g'(r)=\p_rv_i(r\sigma)$ this gives
    \begin{align}
        &\int_0^\infty\left(|\p_rv_i(r\sigma)|^2
        +\frac{k_w}{r^2}|v_i(r\sigma)|^2\right)
        r^{D-1}e^{-\frac{r^2}{2}}dr \notag\\
        &\qquad\geq
        s_w\int_0^\infty|v_i(r\sigma)|^2r^{D-1}e^{-\frac{r^2}{2}}dr.
        \label{ICDhelper6}
    \end{align}
    Summing \eqref{ICDhelper6} over $i$ and using
    $|\p_rv|^2=\sum_i|\p_rv_i|^2$ and $|v|^2=\sum_i|v_i|^2$ we get
    \begin{align}
        &\int_0^\infty\left(|\p_rv(r\sigma)|^2
        +\frac{k_w}{r^2}|v(r\sigma)|^2\right)
        r^{D-1}e^{-\frac{r^2}{2}}dr \notag\\
        &\qquad\geq
        s_w\int_0^\infty|v(r\sigma)|^2r^{D-1}e^{-\frac{r^2}{2}}dr.
        \label{ICDhelper7}
    \end{align}
    Inserting \eqref{ICDhelper7} into \eqref{ICDhelper5} and integrating in
    $\sigma$ against $w(\sigma)d\sigma$, we obtain by \eqref{eq:polar measure}
    \begin{equation}
        \int_\Sigma\|\nabla^2u\|^2d\mu_w
        -\int_\Sigma\nabla^2\log w(\nabla u,\nabla u)d\mu_w
        \geq
        s_w\int_\Sigma|\nabla u|^2d\mu_w.
        \label{ICDhelper8}
    \end{equation}

    It remains to convert \eqref{ICDhelper8} into a curvature-dimension
    statement. Writing the potential as
    \begin{equation*}
        W=\frac{|x|^2}{2}-\log w,
    \end{equation*}
    we have
    \begin{equation*}
        \nabla^2W=\mathbb{I}-\nabla^2\log w,
    \end{equation*}
    so the Bochner identity gives the pointwise formula
    \begin{align}
        \Gamma_2(u)
        &=\|\nabla^2u\|^2+\nabla^2W(\nabla u,\nabla u) \notag\\
        &=|\nabla u|^2+\|\nabla^2u\|^2-\nabla^2\log w(\nabla u,\nabla u).
        \label{ICDhelper9}
    \end{align}
    Integrating \eqref{ICDhelper9} and applying \eqref{ICDhelper8},
    \begin{align*}
        \int_\Sigma\Gamma_2(u)d\mu_w
        &=\int_\Sigma|\nabla u|^2d\mu_w
        +\int_\Sigma\|\nabla^2u\|^2d\mu_w
        -\int_\Sigma\nabla^2\log w(\nabla u,\nabla u)d\mu_w\\
        &\geq\int_\Sigma|\nabla u|^2d\mu_w+s_w\int_\Sigma|\nabla u|^2d\mu_w\\
        &=(1+s_w)\int_\Sigma\Gamma(u)d\mu_w.
    \end{align*}
    Finally $k_w>0$ forces $\sqrt{(D-2)^2+4k_w}>|D-2|$, hence $s_w>0$ by
    \eqref{eq:sw definition}. Therefore $\mu_w$ satisfies
    $ICD(1+s_w,\infty)$ with $1+s_w > 1$.
\end{proof}
We note that Theorem~\ref{thm: improved ICD} never gives more than $ICD(2,\infty)$, and it gives less when $N\geq2$. Indeed, differentiating Euler's identity $x\cdot\nabla\log w=\alpha$ gives $\nabla^2\log w(x,x)=-\alpha$ (see \eqref{eq: radial hessian of log w}), so testing \eqref{codition: scale invariant hessian lower bound} in the direction $x$ gives $k_w\leq\alpha$. By \eqref{eq:sw definition}, $s_w\leq1$ if and only if $\sqrt{(D-2)^2+4k_w}\leq D$, that is, $k_w\leq D-1=N-1+\alpha$. Hence $s_w\leq1$, with equality if and only if $k_w=N-1+\alpha$, which together with $k_w\leq\alpha$ forces $N=1$. So $s_w<1$ when $N\geq2$, and $s_w=1$ when $N=1$ and $k_w=\alpha$ is the best constant. For the full monomial weight, $k_w=\min_i\alpha_i$, and Theorem~\ref{thm:ICD} gives the better constant $2$ by using the product structure.

On the other hand, the constant $2$ holds for every homogeneous weight if we restrict to radial functions. Using the homogeneity of $w$, we can write $L_w$ in polar coordinates as
\begin{equation}
    L_w = \underbrace{\p_{rr} - \left(r - \frac{D-1}{r}\right)\p_r}_{L_r} + \frac{1}{r^2}L_\sigma, \label{eq: L_w in polar}%
\end{equation}
where the angular operator is
\begin{equation}
    L_\sigma = \Delta_{\mathbb{S}^{N-1}} + \nabla_{\mathbb{S}^{N-1}} \log w \cdot \nabla_{\mathbb{S}^{N-1}}. \label{eq: angular operator}
\end{equation}
In particular, for a radial function $f=f(r)$, $r=|x|$, the angular term disappears and
\begin{equation}
    L_wf = f'' - \left(r-\frac{D-1}{r}\right)f',
    \label{eq: radial reduction}
\end{equation}
which is the one-dimensional monomial operator \eqref{Laguerre} with $\alpha$
replaced by $D-1$. By Section~\ref{sec:spectral}, the radial eigenfunctions of $-L_w$ are therefore
$L_k^{\beta}\left(\frac{r^2}{2}\right)$ with $\beta=\frac D2-1$ and eigenvalues $2k$, $k\geq0$.

\begin{theorem}[$ICD(2,\infty)$ on radial functions]\label{thm: radial ICD}
Let $w$ be homogeneous of degree $\alpha$ on $\Sigma$ with $D = N+\alpha > 1$, and
let $u$ be radial, $u = u(r)$ with $r = |x|$ and $u\in C^2(0,\infty)$, satisfying the decay condition
\begin{equation}
    r^{D-2}e^{-\frac{r^2}{2}}|u'|^2 \longrightarrow 0
    \qquad \text{as } r\to0 \text{ and as } r\to\infty.
    \label{eq: radial decay}
\end{equation}
Then
\begin{equation}
    \int_\Sigma \Gamma_2(u)  d\mu_w \geq 2\int_\Sigma \Gamma(u)  d\mu_w.
    \label{eq: radial ICD}
\end{equation}
\end{theorem}

\begin{proof}
Since $u$ is radial we have $\nabla u = u'(r)\sigma$ and
\begin{equation*}
    \p_{ij}u = u''(r)\sigma_i\sigma_j
    + \frac{u'(r)}{r}\left(\delta_{ij} - \sigma_i\sigma_j\right),
\end{equation*}
so that
\begin{equation}
    \Gamma(u) = |u'|^2, \qquad
    \|\nabla^2u\|^2 = |u''|^2 + (N-1)\frac{|u'|^2}{r^2}.
    \label{eq: radial hessian}
\end{equation}
Next we evaluate $\nabla^2\log w$ in the radial direction. Homogeneity of $w$
gives $x\cdot\nabla\log w = \alpha$, and differentiating this identity yields
\begin{equation*}
    \nabla\log w + \nabla^2\log w  x = 0.
\end{equation*}
Contracting with $x$ we obtain $\nabla^2\log w(x,x) = -x\cdot\nabla\log w = -\alpha$,
hence
\begin{equation}
    \nabla^2\log w(\sigma,\sigma) = -\frac{\alpha}{r^2}.
    \label{eq: radial hessian of log w}
\end{equation}
By the Bochner formula \eqref{ICDhelper9} together with \eqref{eq: radial hessian}
and \eqref{eq: radial hessian of log w},
\begin{align}
    \Gamma_2(u)
    &= |\nabla u|^2 + \|\nabla^2u\|^2 - \nabla^2\log w(\nabla u,\nabla u) \notag\\
    &= |u'|^2 + |u''|^2 + (N-1)\frac{|u'|^2}{r^2} + \alpha\frac{|u'|^2}{r^2} \notag\\
    &= |u''|^2 + \left(1 + \frac{D-1}{r^2}\right)|u'|^2,
    \label{eq: radial gamma2}
\end{align}
which is exactly the one-dimensional monomial expression for $\Gamma_2$ with
$\alpha$ replaced by $D-1$.

Subtracting $\int_\Sigma |u'|^2d\mu_w$ from \eqref{eq: radial gamma2} we see that
\eqref{eq: radial ICD} is equivalent to
\begin{equation}
    \int_\Sigma \left[|u''|^2 + \frac{D-1}{r^2}|u'|^2\right] d\mu_w
    \geq \int_\Sigma |u'|^2  d\mu_w.
    \label{eq: radial ICD reduced}
\end{equation}
By \eqref{eq:polar measure} the integrand of \eqref{eq: radial ICD reduced}
depends only on $r$, so the angular factor $\int_{\Sigma\cap S^{N-1}}w(\sigma)d\sigma$
appears on both sides and cancels. Hence \eqref{eq: radial ICD reduced} reduces to
the one-dimensional statement
\begin{equation}
    \int_0^\infty \left[|u''|^2 + \frac{D-1}{r^2}|u'|^2\right]
    r^{D-1}e^{-\frac{r^2}{2}}dr
    \geq \int_0^\infty |u'|^2 r^{D-1}e^{-\frac{r^2}{2}}dr.
    \label{eq: radial ICD 1d}
\end{equation}
The radial measure $r^{D-1}e^{-\frac{r^2}{2}}dr$ is, up to a constant factor, the one-dimensional
monomial measure $\mu_{D-1}$, so \eqref{eq: radial ICD 1d} is Lemma \ref{lemma1}
with $\alpha$ replaced by $D-1$, applied to $h = u'$. This completes the proof of
\eqref{eq: radial ICD}.
\end{proof}

\begin{corollary}[Sharp radial Poincar\'e inequality]\label{cor: radial poincare}
    Let $w$ be homogeneous of degree $\alpha$ on $\Sigma$ with $D=N+\alpha>1$. Then every radial $u\in W^{1,2}(\Sigma,\mu_w)$ with $\int_\Sigma ud\mu_w=0$ satisfies
    \begin{equation}
        \int_\Sigma |u|^2d\mu_w
        \leq
        \frac{1}{2}\int_\Sigma|\nabla u|^2d\mu_w.
        \label{eq: radial poincare}
    \end{equation}
    The constant $\frac{1}{2}$ is sharp, and equality holds if and only if
    \begin{equation*}
        u(r)=c\left(r^2-D\right)
    \end{equation*}
    for some $c\in\mathbb R$.
\end{corollary}
\begin{proof}
By \eqref{eq:polar measure}, for radial functions, $\mu_w$ reduces to the normalized one-dimensional measure $\mu_{D-1}$ in the variable $r$, and $|\nabla u|=|u'(r)|$. Hence $\operatorname{Var}_{\mu_w}(u)=\operatorname{Var}_{\mu_{D-1}}(u)$ and $\int_\Sigma|\nabla u|^2d\mu_w=\int_0^\infty|u'|^2d\mu_{D-1}$. The claim is therefore Theorem~\ref{poincareineq} with $N=1$ and $\alpha=D-1>0$, including the equality case $\frac{D}{2}-\frac{r^2}{2}=-\frac12(r^2-D)$.
\end{proof}

{\bf Acknowledgment.} The first author was partly supported by a discovery grant from NSERC. The second author was partly supported by a grant from the Simons Foundation. The third author was partly supported by GAANN Fellowship from the Department of Education.

\end{document}